\documentclass{amsart}
\usepackage{amsmath, amsthm, amssymb}
\usepackage{verbatim}
\usepackage{tikz}
\usetikzlibrary{matrix}

\usepackage[mathscr]{eucal} 
\usepackage{mathrsfs} 
\usepackage{cancel}
\usepackage{marginnote}

\renewcommand{\d}{\partial}

\newcommand{\wt}[1]{\widetilde{#1}}

\newcommand{\wh}[1]{\widehat{#1}}
\newcommand{\cw}[1]{\check{#1}}

\newcommand{\eps}{\epsilon}
\newcommand{\veps}{\varepsilon}
\newcommand{\vphi}{\varphi}

\newcommand{\al}{\alpha} 
\newcommand{\ze}{\zeta} 

\newcommand{\ga}{\gamma}
\newcommand{\de}{\delta}
\newcommand{\la}{\lambda}
\newcommand{\om}{\omega}
\newcommand{\te}{\theta}

\newcommand{\sig}{\sigma}

\newcommand{\ka}{\kappa}

\newcommand{\Te}{\Theta}
\newcommand{\Ga}{\Gamma}
\newcommand{\Om}{\Omega}
\newcommand{\De}{\Delta}

\newcommand{\cB}{\mathcal{B}}

\newcommand{\cE}{\mathcal{E}}
\newcommand{\cP}{\mathcal{P}}
\newcommand{\cF}{\mathcal{F}}

\newcommand{\supp}{\mbox{supp }}

\newcommand{\rB}{{\bf B}\,}

\newcommand{\bR}{\mathbb{R}}

\newcommand{\bC}{\mathbb{C}}

\newtheorem{thm}{Theorem}
\newtheorem{prop}[thm]{Proposition}
\newtheorem{lem}[thm]{Lemma}
\newtheorem{cor}[thm]{Corollary}

\theoremstyle{definition}
\newtheorem{defn}[thm]{Definition}
\newtheorem{remark}[thm]{Remark}

\numberwithin{thm}{section}
\numberwithin{equation}{section}

\renewcommand{\[}{\begin{equation}}
\renewcommand{\]}{\end{equation}}

\newcommand{\wed}{\wedge}

\newcommand{\ov}[1]{\overline{#1}}
\newcommand{\ul}[1]{\underline{#1}}

\title[Monge-Amp\`ere equations on Hermitian manifolds with boundary]{The Dirichlet problem for the Monge-Amp\`ere equation on Hermitian manifolds with boundary
}
\author{S\l awomir Ko\l odziej and Ngoc Cuong Nguyen}
\address{Faculty of Mathematics and Computer Science, Jagiellonian University, \L ojasiewicza 6, 30-348 Krak\'ow, Poland}
\email{slawomir.kolodziej@im.uj.edu.pl}
\address{Department of Mathematical Sciences, KAIST, 291 Daehak-ro, Yuseong-gu, Daejeon 34141, South Korea}
\email{cuongnn@kaist.ac.kr}

\begin{document}
	\maketitle
	\begin{abstract} We study weak quasi-plurisubharmonic solutions to the Dirichlet problem for the complex Monge-Amp\`ere equation on a general Hermitian manifold with non-empty boundary.  We prove optimal subsolution theorems: for bounded and H\"older continuous quasi-plurisubharmonic functions.  The continuity of the solution is proved for measures that are well dominated by capacity, for example measures with $L^p$, $p>1$ densities, or moderate measures in the sense of Dinh-Nguyen-Sibony. 
	\end{abstract}
	
\bigskip
\bigskip

\section{Introduction}

{\em Background.} 
The complex Monge-Amp\`ere equation in a strictly pseudoconvex bounded domain has been extensively studied  since  1970's. Bedford and Taylor \cite{BT76} proved the fundamental result on the existence of (weak) continuous plurisubharmonic solutions to the Dirichlet problem with continuous datum. The classical solutions were obtained later by Caffarelli, Kohn, Nirenberg and Spruck \cite{CKNS85} for the smooth boundary condition and smooth positive right hand side. The comprehensive book \cite{GZ-book} contains  results, theirs applications and  references for both - the Dirichlet problem and the  Monge-Amp\`ere equation on compact K\"ahler manifolds without boundary.

On the other hand,  it was discovered independently by Semmes \cite{Se92} and Donaldson \cite{Do99} that the geodesic equation in the space of K\"ahler potentials on a compact K\"ahler manifold is equivalent to the homogeneous Monge-Amp\`ere equation on a compact K\"ahler manifold with boundary, the product of the  K\"ahler manifold and an annulus in the complex plane. In general the solution is at most $C^{1,1}$-smooth, due to an example of Gamelin and Sibony \cite{GS80}. Thus it leads to the study of weak solutions to the equation. The (unique) weak  $C^{1,\bar 1}$-solution was obtained by Chen \cite{Ch00} (with a complement in 
\cite{Bl09}) via a sequence of solutions to non degenerate equations - the so called $\veps$-geodesic equations. To solve the non-degenerate Monge-Amp\`ere equation without pseudoconvexity assumption on the boundary one needs to use ideas from Guan's  \cite{Gu98}. We refer to  the expository paper of Boucksom \cite{Bo12} for more detailed  discussion of those  developments.

Guan and Li \cite{GL10} generalized the smooth subsolution theorem in \cite{Gu98} and \cite{Bl09} to the Hermitian setting. This work occurred amid renewed interest in the complex Monge-Amp\`ere equation on Hermitian manifolds which has been studied earlier by Cherrier \cite{Ch87}. It culminated in the resolution of  the Monge-Amp\`ere equation on compact Hermitian manifolds by Tosatti and Weinkove \cite{TW10b}. Weak solutions theory has been developed in \cite{DK12} and \cite{KN1, KN4, KN2, KN7}, \cite{LPT20}, and recent advances for semi-positive Hermitian forms has been made in \cite{GL21a, GL21b}.

Our goal is to study weak solutions to the Dirichlet problem on a smooth compact Hermitian manifold with non-empty boundary. We consider here very general right hand sides, which are positive Radon measures well dominated by capacity  considered in \cite{Ko98, Ko03}. Notice the K\"aher case is included as a special one,  not yet available in the literature.

\medskip

{\em Results.}
Let $(\ov{M}, \om)$ be a $C^{\infty } $ smooth compact Hermitian manifold of dimension $n$ and with non-empty boundary $ \d M$. Thus   $\ov{M}= M \cup \d M$. Let $\mu$ be a positive Radon measure on $M$. Let us denote by $PSH(M, \om)$  the set of all $\om$-plurisubharmonic ($\om$-psh for short) functions on $M$. Consider $\vphi \in C^0(\d M, \bR)$. We  study the Dirichlet problem
\[\label{eq:DP}\begin{cases} 
	u \in PSH(M, \om) \cap L^\infty(\ov{M}),\\	
	(\om + dd^c u)^n = \mu,\\
	\lim_{z \to x} u(z) = \vphi(x) \quad\forall x \in \d M.	
\end{cases}
\]
Since there is no convexity condition on the boundary, to solve the Dirichlet problem a necessary condition is the existence of a subsolution.

\begin{defn}[Subsolution] \label{def:ss}  Let $\ul u \in PSH(M, \om) \cap L^\infty(\ov M)$ be such that $$\lim_{z\to x} \ul u(z) = \vphi(x) \quad \text{for every } x\in \d M.$$ 
\begin{itemize}
\item
[(a)]  It is called a bounded subsolution for the measure $\mu$  if it satisfies:
$$(\om + dd^c \ul u)^n \geq \mu \quad \text{on } M.$$
\item
[(b)] If a bounded subsolution is also continuous (resp. H\"older continuous) on $\ov M$, then we called it as a continuous subsolution (resp. H\"older continuous subsolution) to $\mu$.
\end{itemize}
\end{defn}

\begin{thm} \label{thm:bounded-ss} Suppose that there exists a bounded subsolution $\ul u \in PSH(M, \om) \cap L^\infty(\ov M)$ for $\mu$. 
 Then, the Dirichlet problem \eqref{eq:DP} has a solution.
\end{thm}

This result is a generalization of the bounded subsolution theorem in a bounded strictly pseudoconvex domain due to the first author \cite{Ko95}. Furthermore, it can be considered as the weak solution version of \cite{Gu98} and \cite{Bl09}, \cite{Bo12}.
Note however that our  proof   does not use those smooth solutions  as approximants of the weak solution. For  a Hermitian form $\om$   one needs to improve the stability estimates in \cite{Ko95, Ko03} and apply them to get the statement.

Next we turn to the study of the continuity of solutions if we further assume that the right hand side is also well dominated by capacity, as in  the first author's \cite{Ko98}.

Recall  the Bedford-Taylor capacity  defined in our context as follows.  For a Borel subset $E\subset M$,
$$	cap_\om(E) := \sup\left\{\int_E (\om + dd^cw)^n : w\in PSH(M, \om), 0\leq w\leq 1\right\}.
$$
Let $h : \mathbb R_+ \rightarrow (0, \infty ) $ be an increasing function such that
\[
\label{eq:admissible}
	\int_1^\infty \frac{1}{x [h(x) ]^{\frac{1}{n}} }  \, dx < +\infty.
\]
In particular, $\lim_{ x \rightarrow \infty} h(x) = +\infty$. Such a function $h$ is called
{\em admissible}. If $h$ is admissible, then so is $A_2 \, h (A_1x)$ for every  $A_1,A_2 >0$.
Define
\[
	F_h(x) = \frac{x}{h(x^{-\frac{1}{n}})}.
\]
 Let $\mu$ be a positive Radon measure satisfying  for some admissible $h$:
\[
\label{eq:dominate}
	\mu(E) \leq F_h( cap_\om (E)),
\]
for any Borel set $E \subset X$. The set of all measures satisfying this inequality for some admissible $h$ is denoted by $\cF(X,h)$.  In what follows  we often omit to mention that $h$ is admissible.

\begin{cor}\label{cor:cont-ss} Consider the subsolution from  Theorem~
\ref{thm:bounded-ss}. Assume that $\mu \in \cF(M,h)$ for an admissible function $h$. Then, the solution is continuous on $\ov M$.
\end{cor}

An important class of such measures are the ones with $L^p$-density with respect to the Lebesgue measure  ($p> 1$)  (Lemma~\ref{lem:sublevel-set}). Another class is  the Monge-Amp\`ere measures of H\"older continuous quasi-plurisubharmonic functions on $\ov M$ (Theorem~\ref{thm:moderate-boundary}). 
It still remains an open problem, even in a strictly pseudoconvex set in $\bC^n$, whether a continuous subsolution leads to the continuous solution.

The best regularity of solutions of the Dirichlet problem in our considerations is the H\"older one (see e.g. \cite{GKZ08}). 
In the compact K\"ahler manifold case it was proved in \cite{Ko08} for $L^p ,\  (p>1)$ right hand side. Then  the H\"older continuous subsolution theorem was proved  in \cite{DDGHKZ}. We prove here a significant generalization of \cite{Ng17, Ng18}, which answered positively a question by Zeriahi in the local setting. That is if the subsolution is  H\"older continuous, then the solution has this property too.

\begin{thm}\label{thm:holder-ss} Assume that the subsolution $\ul u \in PSH(M,\om) \cap C^{0,\al}(\ov M)$ for some $\al>0$. Then, the solution $u$ obtained in Theorem~\ref{thm:bounded-ss} is H\"older continuous on $\ov M$.
\end{thm}

Notice that the H\"older exponent of the solution depends only on the dimension and the H\"older exponent of the subsolution, as in  \cite{DDGHKZ} and \cite{KN4} for compact complex manifolds without boundary.

The uniqueness of a weak solution is still an open problem  in general on compact Hermitian manifolds without boundary (see \cite{KN2} for a partial result). We are able to prove this property under  some  extra assumption on either the metric $\om$ or the manifold.

\begin{cor}[Uniqueness of bounded solution] \label{cor:uniqueness}
Suppose that $M$ is Stein or $\om$ is K\"ahler. Let $u, v$ be bounded $\om$-psh on $M$ such that $\lim\inf_{z \to \d M} (u-v)(z) \geq 0$. Assume that $(\om + dd^c u)^n \leq (\om +dd^c v)^n$ in $M$. Then,
$ u \geq v$ on $\ov{M}$. In particular, there is at most one bounded solution to the Dirichlet problem~\eqref{eq:DP} in this setting.
\end{cor}

\medskip

{\em Organization.} In  Section~\ref{sec:cegrell}  we prove various convergence theorems in the Cegrell class of plurisubharmonic functions. These results combined with the Perron method allow to derive the bounded subsolution theorem (Theorem~\ref{thm:bounded-ss}) in Section~3. 
  In Section~\ref{sec:continuous-ss} we prove the stability estimates for Hermitian manifolds with boundary for measures that are well-dominated by capacity. This is done by adopting the proofs of  stability estimates from the setting of  compact Hermitian manifold without boundary.  Section~\ref{sec:boundary-chart} contains local estimates on interior  and boundary charts. The key technical  result is  the bound on volumes  of sublevel sets of quasi-plurisubharmonic functions in a certain Cegrell class in a boundary chart (Lemma~\ref{lem:sublevel-set}).  In Section~\ref{sec:holder-ss} we prove the H\"older continuous subsolution theorem. One needs to consider a smoothing  of the bounded solution obtained in Theorem~\ref{thm:bounded-ss}, via the geodesic convolution, due to Demailly. Then we obtain the global stability estimate to control the modulus of continuity of the solution. For this we use a rather delicate construction choosing carefully two exhaustive sequences $M_\veps \subsetneq M_\de$ of the manifold $M$, and keeping track of the dependence of the modulus of continuity of the solution in collar sets on both parameters $\veps>\de>0$ (Proposition~\ref{prop:delta}). Finally, we give the proof of  the uniqueness of solution when  either the manifold is  Stein or $\om$ is K\"ahler.

\medskip

{\em Notation.}  The uniform constants $C, C_0, C_1,...$ may differ from line to line. For simplicity we denote $\|\cdot\|_\infty$ to be the supremum norm of functions in the considered domain. We often write $\om_v:=\om + dd^c v$ for a quasi-plurisubharmonic function $v$.

\medskip

{\em Acknowledgements.} The first author is partially supported by  grant  no. \linebreak 2021/41/B/ST1/01632 from the National Science Center, Poland. The second author is  partially supported by the start-up grant G04190056 of KAIST and the National Research Foundation of Korea (NRF) grant  no. 2021R1F1A1048185. We also thank the referees for a very careful reading and constructive comments. One of referee's suggestions shortened our original proof.

\section{Convergence in the Cegrell class}
\label{sec:cegrell}

Let $\Om$ be a bounded strictly pseudoconvex domain in $\bC^n$. Let us recall the Cegrell class introduced in \cite{Ce98}:
$$\cE_0(\Om) =\left\{  u\in PSH (\Om) \cap L^\infty(\Om):  \lim_{z\to \d\Om} u(z) =0, \; \int_\Om (dd^c u)^n <+\infty \right\}.$$

First we prove a bunch of convergence results to be used in the following sections. In what follows, when the domain of integration $\Om$ is fixed and no confusion arises we often write
$$
	\int  g d\la := \int_\Om g d\la
$$
for a Borel function $g$ on $\Om$. For a Borel set $E\subset \Om$, we denote by  $cap(E) := cap(E,\Om)$ its Bedford-Taylor capacity.

The following is implicitly contained in the last part of the proof of \cite[Lemma~5.2]{Ce98}.  Cegrell dealt with sequences from $\cE_0(\Om)$ but the proof works for sequences of  negative plurisubharmonic functions as well. 
\begin{lem} \label{lem:L1-norm-convergence}
Let $d\la$ be a finite positive Radon measure on $\Om$ which vanishes on pluripolar set. Suppose that $u_j \in \cE_0(\Om)$ is a  uniformly bounded sequence that  converges a.e. with respect to the Lebesgue measure $dV_{2n}$ to  $u \in \cE_0(\Om)$. Then there exists a subsequence  $\{u_{j_s}\} \subset \{u_j\}$ such that
$$  \lim_{s\to +\infty} \int_\Om u_{j_s} d\la = \int_\Om u d\la.$$
\end{lem}

\begin{proof} Since $d\la$ is a finite measure it follows that $\sup_{j} \int_{\Om} |u_j|^2 d\la < +\infty$. So there exists a subsequence $\{u_j\}$ weakly converging to $v \in L^2 (d\la)$. By the Banach-Saks theorem we can find a subsequence $u_{j_k}$ such  that 
$$	
	F_k= \frac{1}{k} (u_{j_1} + \cdots + u_{j_k}) \to v \quad\text{in } L^2(d\la)
$$
as $k \to +\infty$. Extracting a subsequence $\{F_{k_s}\}_s$ of $\{F_k\}$ we  get $F_{k_s} \to v$ a.e in $d\la$, and also  that $F_{k_s}$ converges a.e to $u$ with respect to the Lebesgue measure.  Therefore, $(\sup_{s>t} F_{k_s})^* \searrow u$ everywhere as $t\to +\infty$.
It follows that  there is a subsequence which we still denote by $\{u_j\}$ such that
$$
	\lim_{j\to \infty} \int u_j d\la = \int v d\la = \lim_{s\to \infty} \int F_{k_s} d\la = \lim_{t\to \infty} \int \sup_{s>t} F_{k_s}  d\la = \int u d\la,
$$
where the first identity used the decreasing convergence property; the second one used the a.e-$d\la$ convergence, and  the last used the fact that $d\la$ does not charge  pluripolar sets. This completes the proof.
\end{proof}

\begin{cor}\label{cor:L1-convergence}  We keep  the assumptions of Lemma~\ref{lem:L1-norm-convergence}. Assume moreover that $d\la (E) \leq C_0 cap (E)$ for every Borel set $E \subset \Om$ with a uniform constant $C_0$. Then there exists a subsequence, which is still denoted by $\{u_j\}$, such that
$$ \lim_{j\to \infty} \int |u_j - u| d\la  =0.
$$
\end{cor}

\begin{proof} By Lemma~\ref{lem:L1-norm-convergence} we know that 
$$\lim_{j\to \infty} \int u_j d\la = \int u d\la, \quad \lim_{j\to \infty} \int \max\{u_j, u\} d\la = \int_{\Om} u d\la.$$
Fix $a >0$. We have $\{|u-u_j|>a\} = \{u- u_j >a\} \cup \{u-u_j<-a\}$. Therefore
$$
	\int_{\{u-u_j>a\}} d\la = \int_{\{\max\{u_j,u\} - u_j >a\}} d\la \leq \frac{1}{a}  \int_{\Om} \left( \max\{u_j,u\} - u_j \right) d\la \to 0 .
$$
Next note that $\max\{u, u_j\} \to u$ in capacity, by the Hartogs lemma and the quasi-continuity of $u$. It follows that
after using the last assumption
$$
	\int_{\{ u- u_j < -a\}} d\la \leq C_0 cap ( |\max\{u_j, u\} - u| >a) \to 0.
$$
In conclusion we get that $u_j \to u$ with respect to the  measure $d\la$ and $\lim \int |u_j| d\la= \int |u| d\la$. As a byproduct we also get that $u_j \to u$ in $L^1(d\la).$
\end{proof}

\begin{lem} \label{lem:uniform-L1-convergence} Still under the assumptions of Lemma~\ref{lem:L1-norm-convergence} we also suppose that  $\sup_j \int (dd^c u_j)^n \leq C_1$ for some $C_1>0$. Let $w_j \in \cE_0(\Om)$ be a uniformly bounded sequence of plurisubharmonic functions in $\Om$ satisfying $\sup_j \int (dd^c w_j)^n \leq C_2$ for some $C_2>0$. Assume that $w_j$ converges in capacity to $w\in \cE_0(\Om)$. Then, 
$$
	\lim_{j\to \infty}  \int |u- u_j| (dd^c w_j)^n = 0.
$$
\end{lem}

\begin{proof} Note that $|u-u_j| = (\max\{u, u_j\} - u_j) + (\max\{u, u_j\} - u)$. First, as in the proof of Corollary~\ref{cor:L1-convergence} we have $\phi_j := \max\{u, u_j\} \to u$ in capacity. Fix $\veps>0$. Then,  when $j $ is large,
$$\begin{aligned}
	\int _{\Om} (\max\{u, u_j\} -u) (dd^c w_j)^n 
&\leq \int_{\{|\phi_j -u|>\veps\}} (dd^c w_j)^n + \veps \int_{\Om} (dd^c w_j)^n \\
&\leq C_0 cap (|\phi_j -u|>\veps) + C_2 \veps.
\end{aligned}$$
Therefore, $\lim_{j\to \infty}  \int (\phi_j-u) (dd^c w_j)^n =0$. Next, we consider for $j>k$,
$$
\int (\phi_j -u_j) (dd^c w_j)^n - \int (\phi_j - u_j) (dd^c w_k)^n = \int (\phi_j - u_j) dd^c (w_j - w_k) \wed T ,
$$
where $T= T(j,k) = \sum_{s=1}^{n-1} (dd^c w_j)^s \wed (dd^c w_k)^{n-1-s}$. By integration by parts 
$$\begin{aligned}
	\int (\phi_j - u_j) dd^c (w_j - w_k) \wed T 
&= \int (w_j - w_k) dd^c (\phi_j -u_j) \wed T \\
&\leq \int |w_j - w_k| dd^c (\phi_j + u_j) \wed T.
\end{aligned}
$$
Since $\|w_j \|_\infty, \|u_j\|_\infty \leq A$ in $\Om$ it follows that 
$$\begin{aligned}
	 \int_{\Om} |w_j - w_k| dd^c (\phi_j + u_j) \wed T 
&\leq  A \int_{\{|w_j -w_k| >  \veps\}}  dd^c (\phi_j + u_j) \wed T \\ &\quad+ \veps\int_{\{|w_j -w_k| \leq \veps\}}  dd^c (\phi_j + u_j) \wed T  \\
&\leq  A^{n+1} cap( |w_j -w_k| >  \veps)  + 	C \veps,
\end{aligned}$$
where the uniform bound for the second integral on the right hand side followed from \cite{Ce04} (see also Corollary~\ref{cor:mass2} below).
It means that the left hand side is less than $2C \veps$ for some $k_0$ and every $j>k\geq k_0$. Thus, 
$$\begin{aligned}
\int (\phi_j -u_j) (dd^c w_j)^n 
&\leq  \int (\phi_j -u_j) (dd^c w_k)^n \\ 
&\quad +	\left|\int (\phi_j -u_j) (dd^c w_j)^n - \int (\phi_j - u_j) (dd^c w_k)^n \right| \\
&\leq  \int (\phi_j -u_j) (dd^c w_k)^n  + 2C \veps \\
&\leq \int |u-u_j| (dd^c w_{k})^n + 2 C \veps.
\end{aligned}$$
Fix $k=k_0$ and apply Corollary~\ref{cor:L1-convergence} for $d\la = (dd^c w_{k_0})^n$ to get that for $j \geq k_1 \geq k_0$
$$
	 \int (\phi_j -u_j) (dd^c w_j)^n  \leq (2C + 1) \veps.
$$
Since $\veps>0$ was arbitrary, the proof of the lemma is completed.
\end{proof}

\section{Bounded subsolution theorems}
\label{sec:bounded-ss}

Our goal in this section is the proof of Theorem~\ref{thm:bounded-ss}, but first
we shall prove  it in the special case of $M \equiv \Om \subset \bC^n$ - a strictly  pseudoconvex bounded domain. 
Then the general statement will follow from this and the balayage procedure.
Let $\mu$ be a positive Radon measure in $\Om$ and  let $\vphi$ be a continuous function on the boundary $\d\Om$. Assume that $\om$ is a Hermitian form in a neighborhood of $\bar\Om$.

\begin{thm} \label{thm:bounded-subsolution}Suppose that $d\mu \leq (dd^c v)^n$ for some bounded plurisubharmonic function $v$ in $\Om$ with $\lim_{z \to \d \Om} v(z) =0$. Then exists a unique $\om$ -plurisubharmonic function $u\in PSH(\Om , \om) \cap L^\infty(\Om )$ solving 
\[\begin{aligned} 
&	(\om+ dd^c u)^n = d\mu, \\
&	 \lim_{\zeta \to z} u(\zeta) = \vphi(z) \quad \mbox{ for } z\in \d \Om.
\end{aligned}\]
\end{thm}

We begin with showing that it is enough to prove the statement under additional hypothesis on  $\vphi , d\mu $ and $(dd^c v)^n$ .
This  reduction is done in three steps:

{\bf Step 1:} One can assume that $\supp \mu$ is compact in $\Om$.  Indeed, let $u_0 \in PSH(\Om, \om) \cap C^0(\bar\Om)$ be the solution satisfying $u_0 =\vphi$ on $\d\Om$ and
$ (\om + dd^cu_0)^n = 0$ in $\Om$ (it exists thanks to \cite[Corollary~4.1]{KN1}). 
Let $\eta_j$ be a  non-decreasing sequence of cut-off functions such that  $\eta_j \uparrow 1$ on $\Om$. 
Then, the sequence of solution $u_j$'s corresponding to $\mu_j = \eta_j \mu$ is uniformly bounded by $u_0 + v \leq u_j \leq u_0$. Hence, by the comparison principle and the convergence theorem 
(\cite{DK12}, \cite[Corollary~3.4]{KN1})
they will decrease to the solution for $\mu$.

{\bf Step 2:} We may assume further that the boundary data is in $C^2(\d \Om)$. Indeed, suppose that the problem is solvable for $\mu$ with compact support and let $\vphi_k\in C^2(\d\Om)$ be a sequence that decreases to $\vphi\in C^0(\d\Om)$. Then the sequence of solutions $u_k$ to 
$$(\om+ dd^c u_k)^n = \mu, \quad u_k = \vphi_k \quad \text{on } \d\Om$$
is decreasing and uniformly bounded. The limit $u = \lim u_k$ is the required solution for the continuous boundary data.

{\bf Step 3:} Reduction  to the case of $v$ defined in a neighborhood of $\bar\Om$ with $\lim_{z \to \d\Om} v(z) =0$, and the support of $(dd^c v)^n$  compact in $\Om$. We already suppose that $\mu$ has a compact support in $\Om$. Then we can modify the subsolution $v$ so that $v$ is defined in a neighborhood of $\bar\Om$ and it satisfies for any $z\in \d\Om$, 
\[	\lim_{\zeta \to z} v(\zeta) =0.
\] 
By the balayage procedure we may further assume that the support of $(dd^c v)^n$ is compact in $\Om$. 

With the above additional assumptions we proceed to define the expected solution.
For  $v$  as in {\em Step 3}, we consider the standard regularizing sequence $v_j \downarrow v$. Then, we write
$(dd^c v_j)^n = f_j dV_{2n}$. By \cite{Ko96} there exists $\cw v_j \in PSH(\Om) \cap C^0(\bar\Om)$ such that  $\cw v_j =0$ on $\d \Om$, and  
\[
	(dd^c \cw v_j)^n = f_j dV_{2n} \quad\text{in } \Om.
\]
We observe that by the Dini theorem $v_j$ converges to $v$ uniformly on compact sets, where the restriction of $v$ is continuous. Consequently, $\cw v_j$  converges to $v$ on those compact sets because by the stability estimate for the Monge-Amp\`ere equation
$$
	\sup_{\Om} |\cw v_j - v_j| \leq \sup_{\d \Om} |\cw v_j - v_j| = \sup_{\d \Om} |v_j -v|.
$$
Note also that
\[\label{eq:mass-subsolution}
	\int_\Om (dd^c \cw v_j)^n \leq C_1.
\]
This follows from the  compactness of the  support of  $\nu =(dd^c v)^n.$

Let $0\leq h \leq 1$ be a continuous function with compact support in $\Om$.  Notice that 
\[\label{eq:weak-approximation}
	hf_j dV_{2n} \to h (dd^cv)^n \quad \text{weakly}.
\]
We first show the existence of a solution for the measure  $h(dd^c v)^n$ obtained as the limit of solutions of $h f_j dV_{2n}$ for a certain subsequence of $\{f_j \}.$ 
Applying \cite[Corollary~0.4]{KN1} we solve the Dirichlet problem
\[\label{seq}\begin{aligned}
&	u_j \in PSH(\Om, \om) \cap C^0(\bar\Om),  \\
&	(\om+ dd^c u_j)^n = h f_j dV_{2n}, \\
&	u_j(z) = \vphi(z) \quad \mbox{for } z \in \d\Om.
\end{aligned}
\]
We define
$$
	u = (\limsup_{j\to \infty} u_j)^* = \lim_{j\to \infty} (\sup_{\ell \geq j} u_\ell)^*.
$$
By passing to a subsequence we may assume that $u_j \to u$ in $L^1(\Om)$ and $u_j \to u$ a.e. with respect to the Lebesgue measure.  Note also that $u \in PSH (\Om, \om) \cap L^\infty(\Om)$ and
$$
	\lim_{z \to \ze} u(z) = \vphi(\ze) \quad\text{for all } \ze \in \d\Om.
$$
This $u$ will be shown to be the solution we are seeking. To do this we  need to prove several  lemmas.

First observe that $\vphi$ can be extended to a $C^2$ smooth function in a neighborhood of $\bar\Om$. 
This allows to produce  a strictly plurisubharmonic function  $g$  in a neighborhood of $\bar\Om$    such that  $dd^c g \geq \om$  and   $g = -\vphi$ on $\d\Om$. 
Let us use the notation
$$
	\wh u = u + g, \quad \wh u_j = u_j + g.
$$

Using an idea of Cegrell \cite[page 210]{Ce98}  we first show that 
\begin{lem}\label{oo}
$\wh u_j \in \cE_0(\Om)$.
\end{lem}

\begin{proof} By  \cite[Theorem~1.1]{GL10} there exists  $\Phi\in PSH(\Om, \om)$ a $C^2$-smooth function on $\bar\Om$ that solves $(\om + dd^c \Phi)^n \equiv 1$ and $\Phi= \vphi$ on $\d\Om$. 
Then,  $$(\om + dd^c u_j)^n \leq (dd^c \cw v_j)^n \leq (\om + dd^c \cw v_j + dd^c\Phi)^n.$$
By the comparison principle \cite[Corollary~3.4]{KN1} we have $u_j \geq \cw v_j + \Phi$. So, $\wh u_j = u_j +g\geq \cw v_j + \Phi + g$. Since $\Phi +g \in PSH(\Om) \cap C^2(\bar\Om)$, and equals zero on $\d\Om$ it belongs to $ \cE_0(\Om)$. Thus  $\cw v_j + \Phi +g \in \cE_0(\Om)$, and so the same is true about  $\wh u_j$.
\end{proof}

\begin{lem}\label{lem:mass1} There exists a uniform constant $C_0$ such that 
$$
	\int_\Om (dd^c \wh u_j)^n \leq C_0.
$$
\end{lem}
 
 \begin{proof}  Set 
$\ga := dd^c g - \om.$
Then $\om + dd^c u_j = dd^c \wh u_j - \ga$.  It follows that 
\[
	 \int_\Om (dd^c \wh u_j -\ga)^n = \int_\Om (\om+ dd^c u_j)^n \leq \int_\Om (dd^c \cw v_j)^n \leq C_1.
\] 
 Using the Newton expansion for the integrand on the left hand side we get that
\[\label{eq:expansion}
	\int_\Om (dd^c \wh u_j)^n - \binom{n}{1} \int_\Om (dd^c \wh u_j)^{n-1} \wed \ga + \cdots + (-1)^n \binom{n}{n}\int_\Om \ga^n \leq C_1.
\]
We are going to show that for $k=1,...,n$,
\[
	\int_\Om (dd^c \wh u_j)^k \wed \ga^{n-k} \leq C_2
\]
for a uniform constant $C_2$.  Indeed, 
since $\ga$ is a smooth $(1,1)$ form in $\bar\Om$, there is a defining function $\psi\in \cE_0(\Om)$ of $\Om$ such that  $\om+ \ga \leq dd^c \psi$ on $\bar \Om$. Using the Cegrell inequality 
\cite[Lemma~5.4]{Ce04} we get  for every $k\geq 1$ 
\[\begin{aligned} \label{eq:lower-term}
	\int_\Om (dd^c \wh u_j )^k \wed \ga^{n-k} 
&\leq		\int_\Om (dd^c \wh u_j)^k \wed (dd^c \psi)^{n-k} \\
&\leq 	\left(\int_\Om (dd^c\wh u_j)^n\right)^\frac{k}{n} \left(\int_\Om (dd^c \psi)^n\right)^\frac{n-k}{n}.
\end{aligned}\]
If we write $m_j^n = \int (dd^c \wh u_j)^n$, it follows from \eqref{eq:expansion} and \eqref{eq:lower-term} that
$$
	m_j^n - const.\binom{n}{k} \sum_{n > k \text{ odd}} m_j^k \leq C_3 \quad \text{for all } j.
$$
Therefore, the total mass of $(dd^c \wh u_j)^n$ is bounded by a uniform constant independent of $j$. Consequently, 
\[\notag	\int_\Om (dd^c\wh u_j)^k \wed \ga^{n-k} \leq C_4
\]
follows by \eqref{eq:lower-term}.\end{proof}

We have also a more general statement. 

\begin{cor} \label{cor:mass2} There exists a uniform constant $C$ such that 
\[	\int_\Om T \leq C
\]
where $T$ are wedge products of $dd^c \wh u_j$ and $dd^c \cw v_k$. 
\end{cor}

\begin{proof} By an application of Cegrell's inequality \cite{Ce04} for every $k\geq 1$, 
\[	\int_\Om (dd^c  \cw v_j)^k  \wed (dd^c \psi)^{n-k}\leq C_4,
\]
where $\psi$ is a strictly plurisubharmonic function defining function of $\Om$ as in the proof of the above lemma.
Now using Cegrell's inequality one more time,
$$\begin{aligned}
 \int_\Om T 
&= 	\int_\Om (dd^c \wh u_j)^p \wed ( dd^c \cw v_k)^q \wed (dd^c \psi)^{n-p-q} \\
&\leq 	\left[I(\wh u_j) \right]^\frac{p}{n} \left[ I( \cw v_k)\right]^\frac{q}{n} \left[I (\psi)\right]^\frac{n-p-q}{n},
\end{aligned}
$$
where $I(w) = \int_\Om (dd^c w)^n$. Finally, all three factors on the right hand side are   bounded by \eqref{eq:mass-subsolution}, Lemma~\ref{lem:mass1} and the smoothness of $\psi$ on $\bar\Om$. 
\end{proof}

\begin{lem} \label{lem:3convergence} Let $ \{u_j\} \subset PSH (\Om , \om )$ be a uniformly bounded subsequence of functions satisfying $u_j(z) = \vphi(z) \quad \mbox{for } z \in \d\Om$
and $u_j \to u$ in $L^1(\Om)$ and $u_j \to u$ a.e. with respect to the Lebesgue measure.
Then one can pick a subsequence
$\{u_{j_s}\} $ such that for 
\[ \label{eq:hartogs-s} w_s = \max\{u_{j_s} , u -1/s\}.
\]
the following equalities hold
 \begin{itemize}
\item
[(a)] $ \lim_{s\to +\infty} \int_\Om |u_{j_s} - u|  (\om+dd^cu)^n =0. $
\item
[(b)] $	\lim_{s\to +\infty} \int_\Om |u_{j_s} -u| (\om + dd^c w_s)^n = 0.$
\item
[(c)] $	\lim_{s\to +\infty}\int_\Om |u_{j_s} -u| (\om+ dd^c u_{j_s} )^n =0.$
\end{itemize}
\end{lem}

\begin{proof} Since $u_j - u = \wh u_j - \wh u$, where $\wh u_j, \wh u$ and $\om_u^n$ satisfy the assumption of  Corollary~\ref{cor:L1-convergence}, the proof of  (a) follows. 

By the Hartogs lemma $w_s$ converges to $u$ uniformly on any compact set $E$ such that $u_{|_E}$ is continuous. Combining this and the quasi-continuity of $u$ it follows that $w_s$ converges to $u$ in capacity. Therefore, by the convergence theorems in \cite{BT82} and \cite{DK12},
$$
	(\om+ dd^c u)^n = \lim_{s\to +\infty} (\om + dd^c w_s)^n. 
$$
We observe that $\om_u^n$ is a finite Radon measure in $\Om$. Recall the notation $\wh w_s = w_s + g$ and, $\wh u = u+g$, where $g$ is a strictly plurisubharmonic  defining function for $\Om$ such that $dd^c g \geq \om$ in a neighborhood of $\bar\Om.$
 Since $w_s = u_{j_s}$ in a neighborhood of $\d \Om$, it follows from Stokes' theorem  and Lemma~\ref{lem:mass1} that
$$
	\int_\Om (\om+ dd^c w_s)^n \leq \int_{\Om} (dd^c \wh w_s)^n =\int_{\Om} (dd^c \wh u_{j_s})^n \leq C_0. 
$$
Letting $s\to+\infty$ we get that $\int\om_{u}^n$ is finite,  and thus $\wh u \in \cE_0(\Om)$. 
Summarizing, $(\om + dd^cw_s)^n \leq (dd^c \wh w_s)^n$ and  $\wh w_s \to \wh u \in \cE_0(\Om)$ in capacity. Hence $(b)$ follows from Lemma~\ref{lem:uniform-L1-convergence}. 

Similarly, with the notation from Lemma~\ref{oo}, $(\om + dd^c u_j)^n \leq (dd^c \cw v_j)^n$ and $\cw v_j $ converges to $v$ in capacity, thus the proof of (c) follows.
\end{proof}

\begin{lem}\label{cor:weak-convergence} 
 Consider $\{ u_j \}$ from the previous lemma. 
Then for a suitably chosen subsequence $\{u_{j_s}\} \subset \{u_j\}$ we have
$$(\om + dd^c u_{j_s})^n \to (\om + dd^c u)^n \quad \text{weakly}.$$
\end{lem}

\begin{proof} By Lemma~\ref{lem:3convergence} we can choose  a subsequence $\{u_{j_s}\} \subset \{u_j\}$ so that $$\int |u-u_{j_s}| (\om + dd^c u_{j_s})^n + \int |u-u_{j_s}|(\om + dd^c w_s)^n < 1/s^2.$$ 
Recall from \eqref{eq:hartogs-s} that $w_s = \max\{u_{j_s}, u-1/s\}$. Then
$$
	{\bf 1}_{\{u_{j_s} > u-1/s\}} (\om+ dd^c w_s)^n = {\bf 1}_{\{u_{j_s} > u-1/s\}} (\om + dd^c u_{j_s})^n.
$$ 
Therefore,  for $\eta \in C^\infty_c(\Om)$,
$$\begin{aligned}
	\left| \int \eta \om_{u}^n - \int \eta \om_{u_{j_s}}^n\right| 
&\leq  \left| \int \eta \om_{u}^n - \int \eta \om_{w_s}^n\right| + \left| \int \eta \om_{w_s}^n - \int \eta \om_{u_{j_s}}^n\right| \\
&\leq \left| \int \eta \om_{u}^n - \int \eta \om_{w_s}^n\right| + \left| \int_{\{u_{j_s} \leq u-1/s \}} \eta \om_{w_s}^n - \eta \om_{u_{j_s}}^n\right|.
\end{aligned}$$
The first term on the right hand side goes to zero as $\om_{w_s}^n \to \om_u^n$. It remains to estimate the second term. Firstly, by the choice of $\{u_{j_s}\} $ at the begining of this proof,
$$\begin{aligned}
	\left|\int_{\{u_{j_s} \leq u-1/s \}} \eta \om_{u_{j_s}}^n \right| 
&\leq \|\eta\|_{L^\infty} \int_{\{u_{j_s} \leq u-1/s \}} \om_{u_{j_s}}^n \\
&\leq s \|\eta\|_{L^\infty} \int |u- u_{j_s}| \om_{u_{j_s}}^n  \leq \frac{1}{s}\|\eta\|_{L^\infty} \to 0 \quad\text{as } s\to +\infty.
\end{aligned}$$
Similarly, 
$$ \begin{aligned}
 \left|\int_{\{u_{j_s} \leq u-1/s\}} \eta \om_{w_s}^n \right| 
 &\leq \|\eta\|_{L^\infty} \left| \int_{\{u_{j_s}\leq  u-1/s\}} \eta \om_{w_s}^n\right|  \\
 &\leq s \|\eta\|_{L^\infty} \int |u- u_{j_s}| \om_{w_s}^n  \to 0 \quad\text{as } s\to +\infty.
\end{aligned}
$$
The last two estimates complete the proof. 
\end{proof}

\begin{proof}[End of proof of Theorem~\ref{thm:bounded-subsolution}]

Now we come back to the sequence defined in \eqref{seq} and its limit $u$.
Let $\{u_{j_s}\}$ be a subsequence of $\{u_j\}$ whose $L^1$-limit and pointwise almost everywhere limit is equal $u$.

Let $0\leq h \leq 1$ be a continuous function with compact support in $ \Om$. Then,  applying the last lemma and \eqref{eq:weak-approximation}, there exists a unique  solution $u\in PSH(\Om, \om) \cap L^\infty(\bar\Om)$ to $$(\om+ dd^cu)^n = h (dd^c v)^n, 
\quad u = \vphi \quad\text{on }\d\Om.$$

By the Radon-Nikodym theorem $d\mu = h d\nu$ for some Borel function $0\leq h \leq 1$. Since $h\in L^1(d\nu)$ and $C_c(\Om)$ is dense in $L^1(d\nu)$, there exists a sequence of continuous functions $0\leq h_k \leq 1 ,$ whose supports are compact in $\Om$,  such that
$$
	\lim_{k \to +\infty} \int |h_k -h| d\nu =0.
$$ 
In particular, $h_k d\nu \to h d\nu$ weakly. Applying the argument above for continuous $h$, we can find $u_k \in PSH(\Om, \om) \cap L^\infty(\Om)$ such that $\lim_{z \to \ze} u_k (z) = \vphi(\ze)$ for every $\ze\in\d\Om$ and
$$ (\om + dd^c u_k)^n = h_k (dd^c v)^n = h_k d\nu.
$$

Define 
$$
	u = (\limsup u_k)^*.
$$
Passing to a subsequence we may assume that $u_k \to u$ in $L^1(\Om)$ and converging $a.e$ to $u$ with respect to the Lebesgue measure. 
Again by Lemma~\ref{cor:weak-convergence} there exists a subsequence $\{u_{k_s}\} $ of  $\{ u_k\}$ such that 
$$
	(\om + dd^c u_{k_s})^n \to (\om+ dd^c u)^n \quad \text{weakly}. 
$$
Hence, 
$$
	(\om + dd^c u)^n = \lim_{k_s \to +\infty} h_{k_s} (dd^c v)^n = h d\nu.
$$
The proof is completed.
\end{proof}


\begin{proof}[Proof of Theorem~\ref{thm:bounded-ss}]

Let us proceed with the proof of the subsolution theorem on $\ov M$ a smooth compact Hermitian manifold with boundary. Consider the following set of functions
\[\label{eq:class-B}
	\cB(\vphi,\mu): = \left\{w \in PSH(M, \om) \cap L^\infty(M): (\om + dd^c w)^n \geq \mu, w^*_{|_{\d M}} \leq \vphi \right\},
\]
where  $w^*(x) = \limsup_{M \ni z\to x} w(z)$ for every $x\in \d M$.
Clearly, $\ul{u} \in \cB(\vphi, \mu)$. Let us solve the linear PDE finding  $h_1 \in C^0(\ov{M}, \bR)$ such that 
\[\label{eq:omega-laplace}\begin{aligned}
	(\om + dd^c h_1) \wed \om^{n-1} =0, \\
	h_1 = \vphi \quad\text{on } \d M.
\end{aligned}
\]
Since $(\om +dd^c w)\wed \om^{n-1} \geq 0$ for $w \in PSH(M, \om)$, the maximum principle for the Laplace operator with respect to $\om$  gives 
$$
	w \leq h_1 \quad \text{for all } w  \in \cB(\vphi, \mu).
$$
Set
\[ \label{eq:supremum}
	u (z)  = \sup_{w \in \cB(\vphi, \mu)} w (z)\quad \text{for every } z\in M.
\]
Then, by Choquet's lemma and the fact that $\cB(\vphi, \mu)$ satisfies the lattice property, $u = u^* \in \cB(\vphi, \mu)$. Again by the definition of $u$, we have   $\ul{u} \leq u \leq h_1$. It follows that 
\[\label{eq:boundary} \lim_{z \to x} u(z) = \vphi(x) \quad \text{for every } x\in \d M.
\]

\begin{lem}[Lift] \label{lem:lift}  Let $v \in \cB(\vphi, \mu)$. Let $B \subset\subset M$ be a small coordinate ball (a chart biholomorphic to a ball in $\bC^n$). Then, there exists $\wt v \in \cB(\vphi, \mu)$ such that $v \leq \wt v$  and $(\om + dd^c \wt v)^n = \mu$ on $B$.
\end{lem}

\begin{proof} We implicitly identify $B' \subset \subset B$ with a  small ball in $\bC^n$, with $B'$ also fixed. Note that on $B$ we have $ {\bf 1}_{B'} d\mu \leq (\om + dd^c v)^n$. Since $\ov{B'}$ is compact in $B$,  we can easily get a bounded plurisubharmonic  subsolution with zero boundary value on $\d B$ for ${\bf 1}_{B'} d\mu$. Let $\phi_j \searrow v$ on $\d B$ be a uniformly bounded sequence of continuous  functions.  Theorem~\ref{thm:bounded-subsolution}  gives the existence of solutions:
$$\begin{cases}	
	v_j\in PSH(B, \om) \cap L^\infty(B), \\
	(\om + dd^c v_j)^n = {\bf 1}_{B'}d\mu, \\
	\lim_{z \to x \in \d B} v_j(z) = \phi_j (x),\quad \forall x\in \d B.	 	
\end{cases}
$$
By the comparison principle $v_j$ is a decreasing sequence as $j\to +\infty$ and $v_j \geq v$ on $B$. Set $w = \lim_j v_j$. Then, $w \in PSH(B, \om)$ and $w \geq v$. By the convergence theorem $(\om + dd^c w)^n = {\bf 1}_{B'}d\mu$. 
Note that $\lim_j v_j(x) = v(x)$ for every $x\in \d B$. Hence,
$
	\limsup_{z\to x} w (z) \leq \phi_j(x).
$
Let $j\to +\infty$ we get that $$\limsup_{z\to x} w (z) \leq v(x)$$ for every $x \in \d B$

Now we define 
$$ \wt w = \begin{cases}
	\max\{w, v\} \quad\text{on } B, \\
	v \quad\text{on } M \setminus B.
\end{cases}
$$
Then, $\wt w \in PSH(M, \om ) \cap L^\infty(M)$ and satisfies $\wt w^* \leq \vphi$ on $\d M$. Moreover, 
$$
	(\om + dd^c \wt w)^n \geq d\mu \quad\text{on } (M\setminus B) \cup B'.
$$

Finally, let $B_j \nearrow B$  be a sequence of open balls increasing to $B$, then by the above construction we get a decreasing sequence $\wt w_j \in PSH(M, \om) \cap L^\infty(M)$ such that  $\wt w_j^* \leq \vphi$ on $\d M$  and
$$
	(\om + dd^c w_j)^n \geq d\mu \quad \quad\text{on } (M\setminus B) \cup B_j.
$$
Set $\wt v = \lim w_j \geq v$. Then, $\wt v$ is the required lift function.
\end{proof}

\noindent
{\em End of Proof of the bounded subsolution theorem.}  By \eqref{eq:boundary} it remains to show that the function $u$ above satisfies $(\om +dd^c u)^n = \mu$. Let $B \subset\subset M$ be a small coordinate ball. It is enough to check $(\om + dd^c u)^n = \mu$ on $B$. Let $\wt u$ be the lift of $u$ as in Lemma~\ref{lem:lift}. It follows that $\wt u \geq u$ and $(\om + dd^c \wt u)^n = \mu$ on $B$. However, by the definition $\wt u \leq u$ on $M$.  Thus, $\wt u = u$ on $B$, in particular on $B$ we have $(\om + dd^c \wt u)^n  = (\om +dd^c u)^n = \mu$.
\end{proof}

\section{Continuity of the solution}
\label{sec:continuous-ss}

First  we recall some facts in pluripotential theory for a Hermitian background form.
The  proofs from \cite{KN1, KN4, KN2}  can be easily adapted to our setting. 

Let ${\bf B}>0$ be a constant such that on $\ov M$ we have
$$	-\rB \om^2 \leq 2n dd^c \om \leq \rB \om^2, \quad -\rB \om^3 \leq 4n^2 d\om \wed d^c \om \leq \rB \om^3.
$$
Then we have the following version (see  \cite[Theorem~0.2]{KN1}) of the classical comparison principle \cite{BT82}.

\begin{lem}\label{lem:mCP} 
Let $u, v \in PSH(M,\omega)\cap L^\infty(\ov M)$ be such that $u, v \leq 0$.  Assume $\liminf_{z\to \d M} (u-v)(z) \geq 0$ and $-s_0= \sup_{M}
(v-u)>0$. Fix $0 < \te < 1$ and set
$m(\te) = \inf_M [ u - (1-\te) v]$. 
Then for any $0 < s < \min\{\frac{\te^3}{16 \rB}, |s_0| \}$,
\[
	\int_{ \{ u< (1-\te) v+ m(\te) + s\}}
		\omega_{(1-\te) v }^n
	\leq \left(1 +  \frac{s \rB}{\te^n} \, C\right)
	 \int_{ \{ u < (1-\te) v + m(\te) + s\}}
		\omega_{u}^n,
\]
where $C$ is a uniform constant depending only on $n$.
\end{lem}


\begin{cor}[Domination Principle] \label{cor:DP} Let $u, v \in PSH(M, \om) \cap L^\infty(\ov M)$ be such that  $\lim\inf_{z \to \d M} (u-v)(z) \geq 0$. Suppose that $u \geq v$ almost everywhere with respect to $\om_u^n$. Then $u \geq v$ in $M$.
\end{cor}

\begin{proof} See \cite[Lemma~2.3]{Ng16} and \cite[Proposition~2.2]{LPT20}.
\end{proof}

The next one is a generalization of \cite[Theorem~5.3]{KN1} to Hermitian manifolds with boundary.

\begin{thm}
\label{thm:kappa}
Fix $ 0 < \te <1$. Let $ u, v \in PSH (M, \om)\cap L^\infty(\ov M)$ 
be such that $u \leq 0$ and $-1\leq  v \leq 0$. Assume that $\liminf_{z\to \d M} (u-v)(z) \geq 0$ and $-s_0= \sup_{M}(v-u)>0$.  Denote by
$m(\te) = \inf_{M} [ u - (1-\te) v]$, and put
$$
\te_0:=	\frac{1}{3}\min\left\{
	\te^n, 
	\frac{\te^3}{16 \rB}, 
	4 (1-\te) \te^n, 
	4 (1-\te)\frac{\te^3}{16 \rB}, |s_0| \right\}.
$$
Suppose that 
$\omega_u ^n \in \cF(M,h)$. Then, for $0<t< \te_0$,
\[\label{eq:kappa}
	t\leq \kappa\left[ cap_\omega ( U(\te, t))\right],
\]
where $U(\te, t ) = \{ u < (1- \te) v+ m(\te) + t \}$, 
and the function $\kappa $ is defined on the interval $(0,1)$ by the formula
\[\label{eq:kappa2}
	\kappa \left ( s^{-n} \right) 	=	4\, C_n  \left \{
			\frac{1}{ \left [ h ( s )\right]^{\frac{1}{n}} } 
			+ \int_{s}^\infty  \frac{dx}{x \left[ h (x) \right]^{\frac{1}{n}}}			
					\right \},
\]
with a dimensional constant $C_n$.
\end{thm}

We use it to  to generalize the stability estimate  for manifolds with boundary. Let $\hbar(s) $ be the inverse function of $ \ka(s)$ and 
\[\label{eq:Ga}	
\Ga(s) \text{  the inverse function of } s^{n(n+2)+1} \hbar (s^{n+2}).
\]
Notice that
$	\Ga(s) \rightarrow 0 \quad \text{as } s \rightarrow 0^+.
$

\begin{prop}[Stability of solutions]\label{prop:L1-stability} Let $u, v \in PSH(M, \om) \cap L^\infty(\ov M)$ be such that $u, v \leq 0$. Let $\mu \in \cF(M,h)$. Assume that  $\liminf_{z\to \d M} (u-v)(z) \geq 0$ and 
$$
	(\om +dd^c u)^n = \mu .
$$
Then, there exists a constant $C>0$  depending only on $\om$ and $\|v\|_\infty $ such that
\[\notag
	\sup_M(v-u) \leq C\; \Ga\left(\left\|(v-u)_+\right\|_{L^1(d\mu)} \right).
\]
\end{prop}

\begin{proof} It is identical to the one in \cite[Proposition~2.4]{KN4}.
\end{proof}

Now we are going to prove
Corollary~\ref{cor:cont-ss}.  So we assume the existence of a subsolution for the measure $\mu$ and the boundary data $\vphi$ as in  Theorem~
\ref{thm:bounded-ss} and  on top of that suppose $\mu \in \cF(M,h)$ for an admissible function $h$. Then we claim that the solution is continuous.

We argue by contradiction. Suppose $u$ were not continuous. Then the discontinuity of $u$ would occur at an interior point of $M$. Hence
$$ 	d = \sup_{\ov M} (u - u_*) >0,
$$
where $u_*(z) = \lim_{\epsilon \to 0} \inf_{w\in B(z, \epsilon)} u(w)$ is the lower-semicontinuous regularization of $u$. Consider the closed nonempty set
$$ F= \{u- u_* = d\} \subset\subset M.
$$
The inclusion follows from the boundary condition. One can extend the boundary data $\vphi$ to a continuous function on 
$\ov{M}$. Then one can use the boundary condition again and the compactness of $\d M$ to choose $M'\subset\subset M'' \subset\subset M$ with $\ov{M'}$ so close to $\ov M$ (in the sense of the distance induced by the metric $\om$) that
$$
|u-\vphi |<d/4     \text{  on  } M\setminus M' .
$$
By the approximation property of quasi-plurisubharmonic functions  \cite{BK07}, one can find a sequence 
\[\label{eq:approximant-v}	
PSH(M'',\om )  \cap C^\infty(M'') \ni u_j \searrow  u \quad \mbox{ in } M''.
\]
By the Hartogs lemma there exists  $j_0>0$ such that for $j>j_0$
$$
u_j \leq \vphi +d/2    \text{  on  } M''\setminus M' .$$
Then the sets $\{ u< u_j -d/4 \}$ are nonempty and relatively compact in $M'$ for $j>j_0$. Moreover, by subtracting a uniform constant we may assume that $$-C_0 \leq u, u_j \leq 0 \quad \text{on } M.$$

Note that for a Borel set  $E\subset M'$,  $cap_\om(E, M) \leq cap_\om(E, M')$. It follows that $\om_u^n \in \cF(M',h)$. Now we apply Theorem~\ref{thm:kappa} for $u$ and $v:=u_j$ on $M'$  to get a contradiction. In fact, $- m_j=\sup_{M'} (u_j -u) \geq d/4$ for $j>j_0$. Let $0<\veps<d/ 12( 1+A_0)$ and take $\te_0$ 
from Theorem~\ref{thm:kappa}.
For $0< s< \te_0$ we have
$$\begin{aligned}
	U_j(\veps, s) 
	&:= \left\{ u< (1-\veps) u_j + \inf_{M'} (u- (1-\veps) u_j)\right\} \\
	&\subset \left\{ u < u_j +m_j + \veps\|u_j\|_\infty + s \right\}  \\
	& \subset \{u< u_j\},
\end{aligned}$$
where for the last inclusion we used the fact $\|u_j\| \leq A_0$ and $0<s \leq -m_j/3$. Fix $0< s< \te_0$. Then that theorem gives 
$$
	s \leq \ka \left[cap_\om\left(U(\veps, s) \right) \right] \leq \ka[cap_\om (u< u_j)].
$$
This leads to a contradiction since $cap_\om (u< u_j) \to 0$ as $j\to +\infty$.

\begin{remark}\label{rmk:admissible} We only need to assume that for any $M'\subset \subset M$, there exists an admissible function $h'$ (it may depend on $M'$) such that  $\mu \in \cF(M', h')$. Then, the solution $u$ is continuous on $\ov M$. Therefore, if $\mu$ is locally dominated by Monge-Amp\`ere of H\"older continuous plurisubharmonic functions, then we have for a fixed compact set $M' \subset\subset M$, $\mu \in \cF(M', h)$ for some admissible function $h$. This is a simple consequence of a result due to Dinh-Nguyen-Sibony \cite{DNS} (Remark~\ref{rmk:moderate-interior}).  More generally, if the modulus of continuity of $\ul u$ satisfies a Dini-type condition then we also get the same conclusion (see \cite{KN5}).
\end{remark}

We have an immediate consequence of this remark.
\begin{cor}\label{cor:continuity-dns} The bounded solution obtained in Theorem~\ref{thm:bounded-ss} is continuous if the subsolution $\ul u$ is H\"older continuous.
\end{cor}

\section{Estimates on boundary charts}
\label{sec:boundary-chart}

For the proof of the subsolution theorem in the H\"older continuous class we need rather delicate estimates close to the boundary of $\ov M$. In this section we prove estimates in local boundary charts of $\ov M$. 
The estimates on  interior charts are easier and they are only mentioned as remarks after the corresponding estimates for  boundary charts. 

Let $q \in \d M$ be a point on the boundary. Let $\Om$ be a boundary chart centered at $q$. Fix $\rho$ a defining function for $\d M$ in $\Om$. Identifying this chart with a subset of $\bC^n$, we can take as $\Om$  the coordinate "half-ball"  of radius $2R>0$ centered at $q$. More precisely let $B_r \subset \bC^n$ for $r>0$ denote the ball of radius $r$ centered at $0$. Let $\rho: B_{2R} \to \bR$ be a smooth function, with $\rho(0) =0$
 and $d\rho \neq 0$ along $\{\rho = 0\}$. Then, 
$$
	\Om = \{ z \in B_{2R}: \rho(z) < 0\}, \quad \d M \cap \Om = \{z: B_{2R}:\rho(z) =0\} .
$$ 
 The Hermitian form $\om$ is smooth up to the boundary $\d M$,  so we can extend it to a smooth Hermitian form on $\ov{B_{2R}}$.
 Multiplying  $\rho$ by a small positive constant we assume also that
$$
	\om + dd^c \rho >0 \quad\text{and}\quad -1 \leq \rho \leq 0 \quad\text{on } \ov{B_{2R}}.
$$
 Moreover,  the estimates are local, so we can fix a K\"ahler form $\Te = dd^c g_0$, where $g_0= C (|z|^2-(3R)^2)$ and $C>0$ is large enough so that
$$
	\Te \geq \om \quad\text{on } \ov{B_{2R}}.
$$

Suppose that 
\[\label{eq:holder-psh-subsolution}
\psi \in PSH(\Om) \cap C^{0,\al}(\ov\Om) \quad\text{and}\quad \psi\leq 0,\]
where $0<\al \leq 1$.
Our goal is to get the stability estimates for $\om$-psh functions with respect to the Monge-Amp\`ere measure of $\psi$. This will provide the stability estimates on boundary charts under the  assumption that a  H\"older continuous subsolution exists.

Let us recall the following class of domains defined in \cite[Definition~3.2]{CKZ11}: A domain $D \subset \bC^n$ is called  \em quasi-hyperconvex \rm if $D$ admits a continuous negative $\om$-psh exhaustion function $\rho: D \to [-1,0)$.

\medskip

We now give several estimates on the quasi-hyperconvex domain 
$$\Om_R =  \Om \cap B_R.$$The first observation is that the continuous exhaustion function for $\Om_R$,
\[\label{eq:rho-R}
	\rho_R = \max\{\rho, |z|^2 - R^2\} \in PSH(\Om_R,\om)
\]
has finite Monge-Amp\`ere mass on $\Om_R$, i.e.,  
$$	\int_{\Om_R} (\Te + dd^c \rho_R)^n \leq C_0.
$$
In fact, it follows from smoothness of $\rho$ and $|z|^2$ on $B_{2R}$ and the Chern-Levin-Nirenberg inequality (see e.g. \cite[page 8]{K05}) that $\rho_R$ has finite Monge-Amp\`ere mass. More generally,  for $1\leq k\leq n$,
\[\label{eq:CKZ-cor3.5}
	\int_{\Om_R} (\Te+ dd^c \rho_R)^k \wed \Te^{n-k} \leq   C_0 \| g_0+ \rho_R \|_{L^\infty(B_{2R})}^{k}.
\]

Let us consider a class corresponding to  the Cegrell class $\cE_0$ in Section~\ref{sec:cegrell}.
\[\label{eq:E0} 
\cP_0 (\Te) = \left\{ v\in PSH(\Om_R, \Te): v\leq 0, \lim_{z\to \d\Om_R} v(z) =0, \int_{\Om_R} \Te_v^n <+\infty \right\},
\]
where $PSH(\Om_R, \Te)$ denotes the set of all $\Te$-psh functions in $\Om_R$.
We first adopt an inequality due to B\l ocki \cite{Bl93} to our setting.
\begin{lem} \label{lem:Blocki-ineq}
Let $v \in \cP_0 (\Te)$ and $1\leq k \leq n$ an integer. Let $\phi_i \leq 0$, $i=1,...,k$, be plurisubharmonic functions in $\Om_R$. Then
$$
	\int_{\Om_R} (-v)^k dd^c\phi_1 \wed \cdots \wed dd^c \phi_k \wed \Te^{n-k} \leq k! \|\phi_1\|_\infty \cdots \|\phi_k\|_\infty \int_{\Om_R} \Te_v^k \wed \Te^{n-k}.
$$
\end{lem}

\begin{proof} We give the proof in a particular case (which is used below) $\phi_1 =\cdots =\phi_k$. The general case is  similar. Denote by $w_\veps$ the function $ \max\{v, -\veps\}$ for $\veps>0$. Then, $w_\eps \nearrow 0$ as $\veps \searrow 0$. By the Lebesgue convergence theorem 
$$
	\lim_{\veps \to 0} \int_{\Om_R} (w_\veps -v)^k (dd^c \phi)^k\wed\Te^{n-k} = \int_{\Om_R} (-v)^k (dd^c \phi)^k\wed\Te^{n-k}.
$$
Fix small $\veps>0$ and write $w$ instead of $w_\veps$.  Then $v\leq w\in PSH(\Om_R, \Te)$ and $w = v$ in a neighborhood of $\d\Om_R$. We first show that
$$
	\int_{\Om_R} (w-v)^k (dd^c \phi)^k\wed \Te^{n-k} \leq k \|\phi\|_\infty \int_{\Om_R} (w-v)^{k-1} (\Te+dd^cv) \wed (dd^c \phi)^{k-1} \wed \Te^{n-k}.
$$
We have 
$$\begin{aligned}
	-dd^c (w-v)^k 
&= - k(k-1) d(w-v) \wed d^c (w-v) - k (w-v)^{k-1} dd^c (w-v) \\
&=	- k(k-1) d(w-v) \wed d^c (w-v) + k (w-v)^{k-1} (\Te_v	- \Te_w)\\
&\leq k (w-v)^{k-1} \Te_v.
\end{aligned}$$
By integration by parts
$$
	\int_{\Om_R} (w-v)^k (dd^c \phi)^k\wed\Te^{n-k} = \int_{\Om_R} \phi dd^c (w-v)^k \wed (dd^c\phi)^{k-1}\wed \Te^{n-k}.
$$
Since $\phi$ is negative in $\Om_R$, it follows from the previous inequality that
$$\begin{aligned}
	\int_{\Om_R} (-\phi) [- dd^c(w-v)^k] &\wed (dd^c \phi)^{k-1} \wed\Te^{n-k} \\
&\leq k \int_{\Om_R} (-\phi) (w-v)^{k-1} \Te_v \wed (dd^c \phi)^{k-1}\wed \Te^{n-k} \\
&\leq k \|\phi\|_\infty \int_{\Om_R} (w-v)^{k-1} (dd^c\phi)^{k-1}\wed \Te_v \wed\Te^{n-k}.	
\end{aligned}$$
So the claim above follows.

Repeating this process $k$-times we obtain
$$
	\int_{\Om_R} (w-v)^k (dd^c\phi)^k\wed\Te^{n-k} \leq k! \|\phi\|_\infty^k \int_{\Om_R} \Te_v^k \wed \Te^{n-k}.
$$
The lemma follows by letting $\veps\to 0$.
\end{proof}

Thanks to this lemma and the quasi-hyperconvexity of $\Om_R$ (every bounded smooth domain in $\bC^n$ is quasi-hyperconvex) we are able to estimate the  Monge-Amp\`ere mass of a bounded plurisubharmonic function on 
\[\label{eq:D-eps}
	D_\veps (R)= \{ z \in \Om_R: \rho_R(z) < -\veps\},  \ \ \ \    \veps>0 .
\]

\begin{cor} \label{cor:mass-decay} Let $1\leq k \leq n$ be an integer. Let $\phi\in PSH(\Om_R) \cap L^\infty(\Om_R)$ be such that $\phi\leq 0$. Then,
$$
	\int_{D_\veps(R)} (dd^c \phi)^k \wed \Te^{n-k} \leq \frac{C \|\phi\|_\infty^k}{\veps^k},
$$
where $C$ is a uniform constant independent of $\veps$.
\end{cor}

\begin{proof}  On $D_\veps(R)$ we have $\max\{\rho_R, -\veps/2\} - \rho_R \geq \veps/2$. It follows from Lemma~\ref{lem:Blocki-ineq} that
$$\begin{aligned}
	\int_{D_\veps(R)} (dd^c \phi)^k \wed \Te^{n-k} 
&\leq \left(\frac{2}{\veps}\right)^k \int_{\Om_R} (\max\{\rho_R, -\veps/2\}) - \rho_R)^k (dd^c \phi)^k \wed \Te^{n-k} \\
&\leq \left(\frac{2}{\veps}\right)^k \|\phi\|_\infty^k \int_{\Om_R} \Te_{\rho_R}^k \wed \Te^{n-k}.
\end{aligned}$$
The last integral is bounded thanks to \eqref{eq:CKZ-cor3.5}.
\end{proof}

\begin{remark} \label{rmk:gradient}
We observe that for a bounded plurisubharmonic function
$$
	2dw \wed d^c w = dd^c w^2 - 2 w dd^c w.
$$
Then applying the above corollary to both terms on the right hand side one obtains
\[\label{eq:gradient-bound}
	\int_{D_\veps(R)} dw \wed d^c w \wed (dd^c \phi)^{k-1}\wed \Te^{n-k} \leq \frac{C \|w\|_\infty^2 \|\phi\|_\infty^{k-1}}{\veps^{k}}.
\]
\end{remark}

The H\"older continuity of a function can be detected by checking the speed of convergence of regularizing sequences
which we now define.
Let us use the following the standard smoothing kernel $\wt\chi (z) = \chi(|z|^2)$ in $\bC^n$, where
\[\label{eq:molifier}
\chi(t)=\begin{cases}\frac {c_n}{(1-t)^2}\exp(\frac 1{t-1})&\ {\rm if}\ 0\leq t\leq 1,\\0&\
{\rm if}\ t>1\end{cases}
\]
 with the constant $c_n$ chosen so that
\begin{equation}\label{eq:total integral}
\int_{\mathbb C^n}\chi(\|z\|^2 )\,dV_{2n}(z)=1 ,
\end{equation}
and $dV_{2n}$ denoting the Lebesgue measure in $\mathbb C^n$.

Let $u \in PSH(\Om) \cap L^\infty(\Om)$ and $\de >0$. 
For $z \in \Om_\de:= \{z\in \Om : {\rm dist}(z, \d\Om) > \de\}$ define
\begin{align}
&\label{eq:cov-loc}	u *\chi_\de(z) = \int_{|x|<1} u(z + \de x) \wt\chi(x) dV_{2n}(x), \\
&\label{eq:av-loc}	\cw u_\de (z) = \frac{1}{v_{2n} \de^{n}} \int_{B_\de(z)} u (x) dV_{2n}(x),
\end{align}
where  $B_\de (z) = \{x \in \bC^n: |z-x| < \de\}$. 
Let us denote  by $\sig$ the surface measure  on a sphere $S(z,r)$ and by  $\overline{\sig }_{2n-1}$ the area of the unit sphere. Consider the averages
$$
	\mu_S (u;z, \de) = \frac{1}{\overline{\sig }_{2n-1} \de^{2n-1}} \int_{S(z, \de)} u(x) d\sig_{2n-1},
$$
which is are increasing in $\de$. 
Therefore, for any $z \in \Om_\de$, 
$$\begin{aligned}
	u*\chi_\de(z) - u(z) 
&= 	\overline{\sig}_{2n-1} \int_0^1 [\mu_S (u; z, \de t)-u(z)] \chi(t^2) t^{2n-1} dt \\
&\leq \| \chi\|_{L^\infty}  \int_0^1 [\mu_S (u; z, \de t)-u(z)] t^{2n-1} dt \\
&= \| \chi\|_{L^\infty}\frac{(\cw u_\de  -u)(z)}{2n} .
\end{aligned}$$
On the other hand, since $\chi \geq 1/C$ on $[0,1/2]$, 
$$\begin{aligned}
	\int_0^1 [\mu_S (u; z, \de t)-u(z)] \chi(t^2) t^{2n-1} dt 
&\geq \frac{1}{C} \int_0^{1/2}  [\mu_S (u; z, \de t)-u(z)] t^{2n-1} dt  \\
& = \frac{1}{2^{n}C} \int_0^1 [\mu_S (u; z, \frac{\de s}{2})-u(z)] s^{2n-1}ds \\
&= \frac{1}{2^{n}C}   \frac{(\cw u_{\de/2}  -u)(z)}{2n}.
\end{aligned}$$
We conclude that there exist a uniform constant $C>0$ such that in $\Om_\de$
\begin{align} \label{lem:cov-reg-compare}
	u *\chi_\de - u &\leq  C (\cw u_\de  -u), \\
	\cw u_{\de/2}  -u &\leq C (u *\chi_\de - u) \label{eq:cov-reg2}.
\end{align}

\begin{lem}\label{lem:boundary-L1}  Let $u \in PSH(\Om) \cap L^\infty(\Om)$. 
 For $0< \de < \de_0 \leq R/4$,
$$	\int_{B_{R/2} \cap \Om_{2\de}} (\cw u_\de (z) - u(z)) dV_{2n} \leq C \de^2 \int_{ B_{3R/4}\cap \Om_\de } \De u(z) dV_{2n}.
$$
Consequently, 
$$
	\int_{B_{R/2} \cap \Om_{\de}} (\cw u_\de (z) - u(z)) dV_{2n} \leq C \de.
$$
\end{lem}

\begin{proof} The first inequality follows from the classical Jensen formula  (see e.g. \cite[Lemma~4.3]{GKZ08}. For the second one we observe
$$\begin{aligned}
	\int_{B_{R/2} \cap \Om_{\de}} (\cw u_\de (z) - u(z)) dV_{2n} 
&\leq  \int_{B_{R/2} \cap \Om_{2\de}} (\cw u_\de (z) - u(z)) dV_{2n} \\&\quad+ \int_{\Om_\de \setminus  \Om_{2\de}} (\cw u_\de (z) - u(z)) dV_{2n} \\
&\leq  C \de^2 \int_{ B_{3R/4}\cap \Om_\de } \De u(z) dV_{2n} + C \de.
\end{aligned}$$ 
Since $\d \Om$ is a  smooth manifold defined by $\rho$ in $B_{2R}$, there exists a uniform constant $c_0>0$ such that $|\rho(z)| \geq c_0{\rm dist}(z, \d\Om)$ for every $z \in B_{R} \cap \Om$. Hence, there exists a uniform constant $c_1 = c_1(c_0, R)>0$ such that $B_{3R/4} \cap \Om_\de \subset  \{\rho_R < -c_1\de\}$ for every $0<\de < \de_0$. Now,  we apply Corollary~\ref{cor:mass-decay} with $\veps= c_1\de$ and $k=1$ for the first term on the right and the proof is completed.
\end{proof}

\begin{remark} \label{rmk:laplacian-interior-chart} In the interior chart the corresponding (stronger) inequality is given by \cite[Lemma~4.3]{GKZ08}. Let $u \in PSH \cap L^\infty (B_{2R})$. Then,  for $0< \de < R/2$,
\[	\int_{B_R} (\cw u_\de - u) dV_{2n} \leq  c_n \left( \int_{B_{3R/2}} \De u dV_{2n}\right) \de^2.
\]
\end{remark}

We also  need various capacities in the estimates.   Consider  the local  $\Te$-capacity:
$$
	cap_\Te (E,\Om_R) = \sup\left\{\int_E (\Te+ dd^c v)^n:  v \in PSH(\Om_R, \Te), \quad -1\leq v \leq 0\right\}.
$$
This capacity  is equivalent to the Bedford-Taylor capacity.
\begin{lem}\label{lem:cap-equiv} There exists a constant $A_1>0$ such that for every Borel set $E \subset \Om$,
$$
	\frac{1}{A_1} cap (E, \Om_R) \leq cap_\Te(E,\Om_R) \leq A_1 cap(E,\Om_R).
$$
\end{lem}

\begin{proof} The first inequality is straightforward while the second inequality follows from the fact that $\Te = dd^c g_0 = dd^c [C (|z|^2-(3R)^2]$ on $\ov\Om$ for some $C>0$ large enough depending only on $\om$ and $\Om$. 
\end{proof}

The following estimate of volumes of sublevel sets  (comp. \cite[Lemma~4.1]{K05}) will be crucial for proving  the volume-capacity inequality on a smoothly bounded domain (without any pseudoconvexity assumption).

\begin{lem} 
\label{lem:sublevel-set}    \rm
There exist  constants $C_0, \tau_0>0$  such that for every $v\in \cP_0(\Te)$ with  $a^n :=  \int_{\Om_R} \Te_v^n$  and  $s\geq 0$, 
$$
	V_{2n} (v<-s) \leq C_0 e^{-\tau_0 s /a}.
$$
\end{lem}

\begin{proof} Since the domain $\Om_R$ is fixed, for any Borel set $E \subset \Om_R$ we will only write $cap_\Te(E) = cap_\Te(E, \Om_R)$ and $cap(E) = cap(E,\Om_R)$ in the proofs.

First we show that there exists a uniform constant $C>0$ independent of $v$ such that for $s\geq 0$,
\[\label{eq:decay-sublevel-set}
	cap_\Te(v<-s) \leq \frac{C a^n}{s^n}.
\]
In fact, let $h \in PSH(\Om_R, \Te)$ be such that $-1 \leq h \leq 0$. 
Note that $\Te = dd^c g_0$ for a strictly smooth plurisubharmonic function $g_0$ in a neighborhood of $\ov{\Om}$ and $g_0\leq 0$.
Then for $\phi = h + g_0$,
$$
	\int_{\Om_R} (-v)^n (\om +dd^ch)^n \leq \int_{\Om_R} (-v)^n (dd^c \phi)^n. 
$$
By Lemma~\ref{lem:Blocki-ineq} it follows that
$$
	\int_{\Om_R} (-v)^n  (dd^c\phi)^n\leq  n! \|\phi\|_\infty^n \int_{\Om_R} \Te_v^n.
$$
Since $\|\phi\|_\infty \leq 1 + \|g_0\|_\infty$,
$$
	\int_{\{v<-s\}} (\Te+ dd^c h)^n 
	\leq \frac{1}{s^n} \int_{\Om_R} (-v)^n (\Te + dd^c h)^n
	\leq \frac{C a^n}{s^n}.
$$
Taking supremum on  the left hand side over $h$ we get the desired inequality. 

Next, denoting $\Om_s:=\{v<-s\} \subset \Om_R$  by an observation in \cite[Proposition~3.5]{Ng14} we know  that
$$
	V_{2n}(\Om_s) \leq C_1 \exp\left( \frac{-\tau_1}{[cap(\Om_s)]^\frac{1}{n}}\right)
$$
for uniform constants $C_1, \tau_1>0$. From  this and  the equivalence between $cap(\bullet)$ and $cap_\Te(\bullet)$ (Lemma~\ref{lem:cap-equiv}) we have
$$
V_{2n}(\Om_s) \leq C_1 \exp\left( \frac{-\tau_1}{[A_1cap_\Te(\Om_s)]^\frac{1}{n}}\right).
$$
Combining this and \eqref{eq:decay-sublevel-set} we get
$V_{2n}(\Om_s) \leq C_1 e^{-\tau_0 s/a}$ with $\tau_0 = \tau_1/(CA_1)^\frac{1}{n}$.
\end{proof}

By comparing the Monge-Amp\`ere measure of a H\"older continuous plurisubharmonic function and  the one of its convolution we show that the estimate of measures of sublevel sets  also holds. 

\begin{prop} \label{prop:H-sublevel-set}Denote by $\mu= (dd^c \psi)^n$, where  $\psi \in PSH(\Om) \cap C^{0,\al}(\ov\Om)$ as in \eqref{eq:holder-psh-subsolution}. Let $\tau_0>0$ be the uniform exponent in Lemma~\ref{lem:sublevel-set}. There exists $\wt\tau_0 =  \wt\tau_0(n, \al, \tau_0)>0$ such that for  $v \in \cP_0(\Te)$ with $a^n = \int_{\Om_R} \Te_v^n \leq 1$ and for every $s>0$, 
$$
	\mu (v < -s ) \leq \frac{C_0 e^{-\wt\tau_0 s /a}}{s^{n+1}} .
$$
\end{prop}

\begin{proof} Let $v_s = \max\{v, -s\}$. Then,  $v_s \in \cP_0(\Te)$, and by \cite[Lemma~3.4]{CKZ11} we also know that $\int_{\Om_R} \Te_{v_s}^n \leq \int_{\Om_R}\Te_v^n = a^n.$
We are going to show that there are uniform constant $0<\al_n  \leq 1$ and $C$ independent of $s$ and $v$ such that
\[\label{eq:holder-functional}
	\int_{\Om_R} (v_s -v) (dd^c \psi)^n \leq \frac{C}{s^n} \left( \int_{\Om_R} (v_s - v) dV_{2n} \right)^{\al_n}.
\]

Suppose this is true for a moment, and let us finish the proof of the proposition. By the inequality $0 \leq v_s - v \leq {\bf 1}_{\{v<-s\}} |v| \leq {\bf 1}_{\{v<-s\}} e^{-\tau v}/\tau$ for every $\tau>0$ (to be determined later), it follows that 
\[\label{eq:vol-cap-c}\begin{aligned}
	\int_{\Om_R} |v_s -v| dV_{2n}
&\leq 	\frac{1}{\tau}\int_{\{v<-s\}} e^{-\tau v} dV_{2n} \\
&\leq 	\frac{1}{\tau}\int_{\Om_R} e^{-\tau v - \tau(v+s)} dV_{2n} \\
&\leq 	\frac{e^{-\tau s}}{\tau} \int_{\Om_R} e^{-2 \tau v} dV_{2n}.
\end{aligned}\]
Using Lemma~\ref{lem:sublevel-set}, for $\tau = \tau_0/(4a)$,  the last integral can be bounded by a uniform constant which is independent of $v$. Moreover, $ v_{s/2} - v \geq \frac{s}{2}$  on  $\{v< -s\} \subset\subset \Omega_R$ for  every $ s> 0$. Therefore,
$$\begin{aligned}
	\mu (v < -s) 
&\leq  \frac{2}{s} \int_{\Om_R} (v_{s/2} -v) (dd^c \psi)^n \\
&\leq  \frac{2C}{s^{n+1}} \left( \int_{\Om_R} (v_{s/2} - v) dV_{2n} \right)^{\al_n},
\end{aligned}$$
where we used \eqref{eq:holder-functional} for the second inequality. Combining this and \eqref{eq:vol-cap-c} we get for every $s>0$, 
$$
	\mu(v<-s) \leq \frac{C}{s^{n+1}} \left(\frac{a e^{-\tau_0 s / (8a)}}{\tau_0} \right)^{\al_n}.
$$
Let $\wt\tau_0 = \tau_0 \al_n/8$. Using $0<a, \al_n \leq 1$, we get
$$	\mu (v<-s) \leq \frac{C e^{-\wt\tau_0s/a}}{s^{n+1}}.
$$
This is the desired estimate.

Now let us prove the promised inequality \eqref{eq:holder-functional}.  For $0 \leq k \leq n$ we write  $S_k := (dd^c\psi)^k \wed \Te^{n-k}$. We show by induction over $k$ that
\[\label{eq:vol-cap-kth} 
	\int_{\Om_R} (v_s -v) (dd^c \psi)^k\wed\Te^{n-k} \leq \frac{C}{s^k} \left( \int_{\Om_R} (v_s - v) dV_{2n} \right)^{\al_k}.
\] 

For $k=0$,  the inequality is obviously true. Suppose it is true for $0\leq k < n$. Denote $T= (dd^c \psi)^k \wed \Te^{n-k-1}$. This means that
$$
	\int_{\Om_R} (v_s -v) T \wed \Te \leq \frac{C}{s^k} \|v_s -v\|_{L^1(\Om_R)}^{\al_k}
$$
for some $0<\al_k \leq 1$.  We need to show that the inequality holds for $k+1 \leq n$ with possibly smaller $\al_{k+1}>0$ and larger $C$. 
Write $\psi_\eps = \psi *\chi_\eps$ on $B_{2R}$  (where we still write $\psi$ for its $\alpha$ - H\"older continuous extension  onto $B_{2R}$) and take $\chi_\eps(z) = \chi(|z|^2/\eps^2)/\eps^{2n}$  the standard smoothing family  defined in \eqref{eq:molifier}.

Firstly, 
\[\label{eq:R-induction-sum}\begin{aligned}
	\int_{\Om_R} (v_s-v) dd^c \psi \wed T 
&\leq \left| \int_{\Om_R} (v_s-v) dd^c \psi_\eps \wed T \right|\\
&+\quad \left|\int_{\Om_R}  (v_s-v) dd^c (\psi_\eps -\psi) \wed T \right| \\
&=:| I_1| + |I_2|. 
\end{aligned}\]
Since $\|\psi\|_\infty \leq 1$,
 $$dd^c \psi_\eps \leq \frac{C \|\psi\|_\infty}{\eps^2} \Te \leq \frac{C \Te}{\eps^2} \quad\text{on } B_{2R}.$$ 
Using this and  the induction hypothesis we get
\[\label{eq:R-i1}
	|I_1| \leq \frac{C \|\psi\|_{\infty}}{\eps^2} \int_{\Om_R} (v_s-v) T \wed \Te \leq \frac{C  \|v_s-v\|_{L^1(\Om_R)}^{\tau_k}}{s^k\eps^2}. 
\]

Next, we estimate $I_2$. By integration by parts we rewrite it as 
$$\begin{aligned}
	I_2 
&=\int_{\Om_R} (\psi_\eps -\psi) dd^c (v_s-v) \wed T \\
&=\int_{\{v\leq -s\}} (\psi_\eps -\psi) (\Te_{v_s}-\Te_v) \wed T \\
&\leq \int_{\{v\leq -s\}} |\psi_\eps -\psi|  (\Te_{v_s}+ \Te_v) \wed T.
\end{aligned}$$
Notice that $v_s - v = 0$ outside  $\{v\leq -s\}$.
By the H\"older continuity of $\psi$ we have $ |\psi_\eps -\psi|  \leq C \eps^\al$. Hence,  
\[\label{eq:R-i2}
	|I_2| \leq C\eps^\al  \int_{\{v\leq -s\}} (\Te_{v_s}+ \Te_v) \wed T.
\]
It follows from Lemma~\ref{lem:Blocki-ineq} that
\[\label{eq:R-i2-a}\begin{aligned}
	\int_{\{v\leq -s\}} \Te_{v_s} \wed T 
&= \int_{\{v\leq -s\}} (dd^c \psi)^k \wed \Te_{v_s} \wed \Te^{n-k-1} \\
&\leq \frac{1}{s^k}\int_{\Om_R} (-v)^k (dd^c \psi)^k \wed \Te_{v_s} \wed \Te^{n-k-1} \\
&\leq \frac{k! \|\psi\|^k_\infty }{s^k} \int_{\Om_R} \Te_v^k \wed \Te_{v_s} \wed \Te^{n-k-1}.
\end{aligned}\]
Similarly,  
\[\label{eq:R-i2-b}
	\int_{\{v\leq -s\}} \Te_{v} \wed T  \leq \frac{k! \|\psi\|^k_\infty }{s^k} \int_{\Om_R} \Te_v^{k+1} \wed \Te^{n-k-1}.
\]
Using  \cite[Corollary~3.5-(3)]{CKZ11} we obtain
\[\label{eq:R-i2-c} \begin{aligned}
 \int_{\Om_R} \Te_v^k &\wed \Te_{v_s} \wed \Te^{n-k-1} \\
&\leq  2^{n-1} \left(  k \int_{\Om_R} \Te_{v_s}^n +  \int_{\Om_R} \Te_{v}^n + (n-k-1)\int_{\Om_R} \Te^n\right) \\
& \leq 2^{n-1} n(a + C_0),
\end{aligned}\]
where $C_0 = \int_{\Om_R} \Te^n$. 
Similarly 
\[ \label{eq:R-i2-d}\int_{\Om_R} \Te_v^{k+1} \wed \Te^{n-k-1} \leq 2^{n-1}n(a + C_0).\]
Combining \eqref{eq:R-i2-a}, \eqref{eq:R-i2-b}, \eqref{eq:R-i2-c} \eqref{eq:R-i2-d} and the assumption $0<a \leq 1$ we get that $$|I_2| \leq \frac{C \eps^\al}{s^k}.$$ 

It follows from the estimates for $I_1$ and $I_2$  that
$$\begin{aligned}
&	\int_{\Om_R} (v_s -v) (dd^c\psi)^{k+1} \wed \Te^{n-k-1}  \\
&=   \int_{\Om_R} (v_s -v ) dd^c \psi \wed T \\
&\leq |I_1| + |I_2| 
\leq \frac{C \|v_s-v\|_{L^1(\Om_R)}^{\al_k}}{s^k\eps^2} + \frac{C \eps^\al}{s^k}.
\end{aligned}
$$
Finally,  we can choose $$\eps = \|v_s-v\|_{L^1(\Om_R)}^{\al_k/3}>0, \quad \al_{k+1} = \al \al_k/3$$ (otherwise $v_s=v$ and the inequality is obvious). Then
$$\begin{aligned}
	\int_{\Om_R} (v_s -v) S_{k+1} 
&= \int_{\Om_R} (v_s -v) (dd^c\psi)^{k+1} \wed \Te^{n-k-1} \\
&\leq \frac{C}{s^k} \|v_s -v\|_{L^1(\Om_R)}^{\al_{k+1}}.
\end{aligned}$$
The proof of the step $(k+1)$ is finished, and so is the proof of  the proposition.
\end{proof}

We state a  volume-capacity  inequality between the Monge-Amp\`ere measure of a H\"older continuous plurisubharmonic function and the Bedford-Taylor capacity on  quasi-hyperconvex domains, where  subsets are of $\veps$-distance  from the boundary.

\begin{thm} \label{thm:moderate-boundary} Let $\psi \in PSH(\Om) \cap C^{0,\al}(\ov\Om)$ as in \eqref{eq:holder-psh-subsolution}. Suppose $0<\veps<R/4$. Then there exist  uniform constants $C, \tau_0>0$ such that for every compact set $K\subset D_\veps (R) \cap B_{R/2}$, 
\[\label{eq:r-vol-cap-inequality}
\int_K (dd^c \psi)^n \leq  \frac{C}{\veps^{n+1}} \exp\left(\frac{-\tau_0 \veps}{[cap(K,\Om_R)]^\frac{1}{n}} \right).\]
In particular, for any $\tau>0$,
$$
	\int_K (dd^c \psi)^n  \leq 	 \frac{C_\tau}{\veps^{n+1}(\tau_0\veps)^{n(1+\tau)}}  [cap(K,\Om_R)]^{1+\tau}  .
$$
\end{thm}

\begin{proof} Our first observation is that we only need to consider subsets satisfying $cap_\Te(K) \leq 1.$ Otherwise, if $cap_\Te(K) >1$, then by Lemma~\ref{lem:cap-equiv}, it follows that $cap(K, \Om_R) \geq 1/A_1$. This implies that  for $0<\tau_0 \leq 2n$,
$$
 \exp\left(\frac{-\tau_0 \veps}{[cap(K,\Om_R)]^\frac{1}{n}} \right) \geq \exp\left(\frac{-\tau_0 \veps}{A_1^\frac{1}{n}} \right) \geq \exp\left(\frac{-n R/2}{A_1^\frac{1}{n}} \right).
$$
Then, the inequality follows from Corollary~\ref{cor:mass-decay} with some  uniform constant $C$. In what follows we work with a subset $E\subset \Om_R$ satisfying
$$
	cap_\Te(E, \Om_R) \leq 1.
$$

Since the domain $\Om_R$ is fixed, we omit it in capacities formulae. We consider yet another capacity which takes into account the geometry of the domain.  Let $E \subset \Om_R$ be a Borel subset,
\[\label{eq:r-cap}
	cap_\rho (E) :=  \sup\left\{ \int_E \Te_w^n : w \in PSH(\Om_R, \Te),\; \rho_R \leq w \leq 0\right\}.
\]
Recall that $-1\leq \rho\leq 0$ is the defining function for $\d M$ on $B_{2R}$. By definition \eqref{eq:rho-R} we have   $-1 \leq \rho_R \leq 0$ in $\Om$. Hence,
\[
	cap_\rho(E) \leq cap_\Te(E).
\]
So $cap_\rho(\bullet)$ does not charge pluripolar sets. Thus, without loss of generality we may assume $K$ is a compact regular subset, in the sense that $h_{\rho, K}$ is continuous in $\Om_R$.
The relative extremal function of this $\rho$-capacity is given by
$$
	h_{\rho, K}(x) = \sup\left\{ w(x): w \in PSH(\Om_R, \Te), \;w_{|_K} \leq \rho_R, w \leq 0  \right\}.
$$
The desired property of $h_{\rho,K}$ is that $\rho_R \leq h_{\rho, K}\leq 0$, which implies that $h_{\rho,K}$ has zero boundary value and $h_{\rho,K} \in \cP_0 (\Te)$.  Moreover, for every compact set $K\subset \Om_R$ the balayage argument shows  $(\Te + dd^c h_{\rho,K})^n \equiv 0$ on $\Om_R \setminus K$. Therefore,
\[\label{eq:rho-capacity-formula}
	cap_\rho(K) \geq \int_{K} (\Te + dd^c h_{\rho, K})^n = \int_{\Om_R} (\Te + dd^c h_{\rho, K})^n.
\]
(This inequality is indeed an identity  we refer the readers to \cite[Proposition~6.5]{BT82} and its generalization in \cite{DL15, DL17}).
 Note that
$$
	K \subset \{h_{\rho,K} = \rho_R\} = \{h_{\rho,K} \leq \rho_R\}.
$$
Hence 
$$
	K \subset \{h_{\rho, K} \leq  \sup_K \rho_R=: \de_K \}.
$$
Write $\mu := (dd^c \psi)^n$. Applying Proposition~\ref{prop:H-sublevel-set} for $h_{\rho,K} \in \cP_0(\Te)$ we get
$$
	\mu(K) \leq \mu(h_{\rho,K} \leq \de_K) \leq \frac{C}{|\de_K|^{n+1}}\exp\left(\frac{-\wt\tau_0 |\de_K|}{a} \right),
$$
where 
$$
	a^n = \int_{\Om_R} (\Te+ dd^c h_{\rho,K})^n \leq cap_\rho(K) \leq cap_\Te(K) \leq 1.
$$
Thus
$$
	\mu (K) \leq \frac{C}{|\de_K|^{n+1}} \exp\left(\frac{-\wt\tau_0 |\de_K|}{[cap_\Te(K)]^\frac{1}{n}} \right).
$$

If $K \subset \{z \in \Om_R : \rho_R(z) < -\veps\} \cap B_{R/2}$, then
$
	|\de_K|  = |\sup_K \rho_R| \geq \veps;
$
hence for such compact sets
\[\label{eq:r-vol-cap-inequality}
\mu (K) \leq \frac{C}{\veps^{n+1}} \exp\left(\frac{-\wt\tau_0 \veps}{[cap_\Te(K)]^\frac{1}{n}} \right).\]
By the equivalence of the capacities  in Lemma~\ref{lem:cap-equiv} the proof of the theorem follows with $\tau_0 = \wt\tau_0 /A_1^\frac{1}{n}$.
\end{proof}

\begin{remark} \label{rmk:moderate-interior} In an interior coordinate chart of $\ov M$ we have the corresponding inequality due to Dinh-Nguyen-Sibony \cite{DNS}.  Let us identify a fixed holomorphic coordinate ball  with $B_{2R}\subset \bC^n$. Suppose that 
\[\label{eq:holder-psh-subsolution2}\notag
\psi \in PSH(B_{2R}) \cap C^{0,\al}(\ov B_{2R})\]
with $0< \al \leq 1$. 
Then, there exist  uniform constants $C = C (R, \psi)$  and $\tau_0 = \tau_0(n, \al)>0$ such that for every compact subset $K \subset B_{R/2}$,
\[\label{eq:dns-inequality}
	\int_K (dd^c \psi)^n 
	\leq C \exp\left(\frac{-\tau_0}{[cap(K,B_{R})]^\frac{1}{n}} \right).
\]
\end{remark}

We are in the position to state the main stability estimate of this section.

\begin{prop}\label{prop:boundary-L1} Take $\psi \in PSH(\Om) \cap C^{0,\al}(\ov\Om)$ as in \eqref{eq:holder-psh-subsolution}.
Assume that $u, v \in PSH(\Om) \cap L^\infty(\Om)$ and $u = v$ on $\{z \in \Om: \rho (z)\geq -\veps\}$.
Then,
$$
	\int_{\Om \cap B_{R/2}}  |u-v| (dd^c \psi)^n \leq \frac{C}{\veps^{n}} \|u-v\|_{L^1(\Om_R)}^{\al_1},
$$
where $C  = C(\Om, R, \psi, \|u\|_\infty, \|v\|_\infty)>0$ and $\al_1 = \al_1 (n, \al)>0$ are uniform constants.
\end{prop}

\begin{proof}
Note that 
$$\begin{aligned}
\d (\Om \cap B_{R/2}) 
&= (\d \Om \cap B_{R/2}) \cup (\Om \cap \d B_{R/2}) \\
& = (\{\rho =0\} \cap B_{R/2}) \cup (\Om \cap \d B_{R/2}).
\end{aligned}$$
By the assumption $u=v$ on $\{z \in \Om: \rho(z) \geq -\veps\}$ the integrand on the left hand side is zero near the first portion of the boundary. Furthermore  $\psi$ is H\"older continuous plurisubharmonic function on $\Om$ a neighborhood of the second boundary portion, so it has finite Monge-Amp\`ere mass by the Chern-Levine-Nirenberg inequality.

By subtracting from $u,v$  a  constant we may assume that $u,v \leq 0$. Also dividing both sides by  $(1 + \|u\|_\infty) (1+ \|v\|_\infty)(1 + \|\psi\|_\infty)^n$
we may assume that $$ -1 \leq u, v, \psi \leq 0 \quad \text{on }\Om.$$

First we suppose $v\geq u$. Let $0\leq \eta \leq 1$ be a cut-off function in $\Om$ such that $\eta \equiv 1$ on $\Om\cap B_{R/2}$ and $\supp \eta \subset \Om \cap B_{3R/4}$. Then $$\supp (\eta (v-u)) \subset D_\veps (R) \subset\subset \Om_R.$$
Thus  it is enough to estimate
$$
	\int_{\Om_R} \eta (v-u) (dd^c\psi)^n .
$$
Let us still write $\psi$ for its H\"older continuous extension of $\psi \in C^{0,\al}$ onto  $B_{2R}$.
 We are going to prove,  by induction over $0\leq k\leq n$,  the inequalities
\[\label{eq:induction-k}
	\int_{\Om_R} \eta (v-u) (dd^c\psi)^k \wed \Te^{n-k} \leq \frac{C}{\veps^k} \|u-v\|_{L^1(\Om_R)}^{\tau_k}.
\]
For $k=n$ it is our statement,  for $k=0$ the inequality holds with $\al_1 =\tau_0=1$. Assume that the inequality is true for $ k\in [0,n)$. 
We need to show it for $k+1 $ with possibly smaller $\tau_{k+1}>0$ and larger $C$.

 Consider the standard  regularizing family  $\chi_t(z) = \chi (|z|^2/t^2)$ for  $0< t < R/4$. Define
$$
	T_k = (dd^c \psi)^k \wed \Te^{n-k} .
$$
Write $\psi_t = \psi *\chi_t$ on $B_{2R}$. Then, for $T = (dd^c\psi)^{k}\wed \Te^{n-k-1}$ we have
\[\label{eq:induction-sum}\begin{aligned}
	\int_{\Om_R} \eta (v-u) dd^c \psi \wed T 
&\leq \left| \int_{\Om_R} \eta (v-u) dd^c \psi_t \wed T \right|\\
&+\quad \left|\int_{\Om_R} \eta (v-u) dd^c (\psi_t -\psi) \wed T \right| \\
&=:| I_1| + |I_2|. 
\end{aligned}\]
Since $\|\psi\|_\infty \leq 1$,
 $$dd^c \psi_t \leq \frac{C \|\psi\|_\infty}{t^2} \Te \leq \frac{C \Te}{t^2} \quad\text{on } B_{2R}.$$ 
Using this and  the induction hypothesis we get
\[\label{eq:i1}
	|I_1| \leq \frac{C \|\psi\|_{\infty}}{t^2} \int_{\Om_R} \eta (v-u) T \wed \Te \leq \frac{C  \|v-u\|_{L^1(\Om_R)}^{\tau_k}}{t^2}. 
\]
Next, by integration by parts we rewrite the second integral in \eqref{eq:induction-sum} as
$$
	I_2 =\int_{\Om_R} (\psi_t -\psi) dd^c (\eta (v-u)) \wed T.
$$
Compute
\[\label{eq:int-by-parts}\begin{aligned}
	dd^c [\eta (v-u)] \wed T 
&= 	(v-u) dd^c \eta \wed T + 2d \eta \wed d^c (v-u) \wed T \\ &\quad+ \eta dd^c (v-u) \wed T.
\end{aligned}\]
Note that $\eta$ is smooth on $\Om$, so 
$dd^c \eta \leq C \Te$.  By the Cauchy-Schwarz inequality
$$\begin{aligned}
&	\left| \int (\psi_t -\psi)d\eta\wed d^c (v-u) \wed T\right|^2 \\ 
&\leq \int_{D_\veps(R)} (\psi_t-\psi)^2d\eta \wed d^c \eta \wed T \int_{D_\veps(R)} d(v-u) \wed d^c (v-u)\wed T.
\end{aligned}$$
Observe that 
$$\begin{aligned}
	|\psi_t (z)- \psi(z)| 
&	\leq \int_{B(0,1)} |\psi(z - tw) - \psi(z)| \chi(|w|^2) dV_{2n}(w)  \\
&	\leq C_\al t^\al
\end{aligned}$$
with $C_\al$  the H\"older norm of $\psi$ on $\ov\Om$,    and $d\eta \wed d^c \eta \leq C_1\Te$,
$$\begin{aligned}
	 \int_{D_\veps(R)} (\psi_t-\psi)^2d\eta \wed d^c \eta \wed T &\leq C_\al C_1 t^{2\al} \int_{D_\veps(R)} T\wed\Te  &\leq \frac{C t^{2\al} \|\psi\|_\infty^{k}}{\veps^{k}},
\end{aligned}$$ 
where we used Corollary~\ref{cor:mass-decay} for the last inequality.
Similarly by Remark~\ref{rmk:gradient},
$$
	\int_{D_\veps(R)} d(v-u) \wed d^c (v-u)\wed T \leq \frac{C \|\psi\|_\infty^{k} \|u\|_\infty^2\|v\|_\infty^2}{\veps^{k+1}}.
$$
For the last term in \eqref{eq:int-by-parts} using Corollary~\ref{cor:mass-decay} again and the H\"older continuity of $\psi$ we have 
$$\begin{aligned}
	\left|\int_{D_\veps(R)} (\psi_t -\psi)\eta dd^c (v-u) \wed T \right|
&\leq \int_{D_\veps(R)}  |\psi_t -\psi|\eta (dd^c u + dd^c v) \wed T \\
&\leq \frac{Ct^{\al} (\|u\|_\infty + \|v\|_\infty) \|\psi\|_\infty^{k} }{\veps^{k+1}}.
\end{aligned}$$
Combining the above estimates  we conclude 
\[\label{eq:i2}
	|I_2| \leq  \frac{C t^\al}{\veps^{k+1}}.
\]
From  \eqref{eq:i1} and \eqref{eq:i2} we get
$$	\int_{\Om_R} \eta (v-u)  (dd^c\psi)^{k+1} \leq |I_1|+ |I_2| \leq \frac{C \|u-v\|_{L^1(\Om_R)}^{\tau_k}}{t^2} + \frac{C t^\al}{\veps^{k+1}}.
$$

If $\|u-v\|_{L^1(\Om_R)}^{\tau_k/4} \geq R/4$, then the inequality of step $(k+1)$  holds for a fixed $t= R/8$ and $\tau_{k+1} = \al \tau_k/4$. On the other hand, we can choose $t = \|u-v\|_{L^1(\Om_R)}^{\tau_k/4}$ and this implies for $\tau_{k+1} = \al \tau_k/4>0$,
$$
	\int_{\Om_R} \eta (v-u)  (dd^c\psi)^{k+1} \leq C \|u-v\|_{L^1(\Om_R)}^{\tau_{k+1}}.
$$
The induction proof is completed under extra hypothesis $v\geq u.$
For the general case  use the identity 
$$
	|u-v| = (\max\{u,v\} - u) + (\max\{u,v\} -v),
$$
and apply the above proof for the pairs ($\max\{u,v\}$, $u$) and ($\max\{u,v\}$, $v$).
\end{proof}

By a  similar (easier) argument for an interior chart of $\ov M$, which we identify with the ball $B_{2R}$ of radius $2R>0$ centered at $0$ in $\bC^n$, we get the following stability estimate.

\begin{lem}\label{lem:interior-L1} Suppose that 
$\psi \in PSH(B_{2R}) \cap C^{0,\al}(\ov B_{2R})$ with $0<\al\leq 1$.
 Assume that $u, v \in PSH(B_{2R}) \cap L^\infty(B_{2R})$. Then,
$$
	\int_{B_{R/2}} |v-u|  (dd^c \psi)^n \leq C  \| u-v\|_{L^1(B_R)}^{\al_1},
$$
where   $C = C(R, \psi, \|u\|_\infty, \|v\|_\infty)$ and $\al_1 = \al_1(n,\al)>0$ are uniform constants.
\end{lem}

\begin{proof} Let $T = (dd^c u)^k \wed (dd^c v)^\ell \wed (dd^c \psi)^{m} \wed \Te^{n-k-\ell-m}$. Since these functions are plurisubharmonic on $B_{2R}$, the Chern-Levine-Nirenberg inequality gives
$$	\int_{B_R} T \leq C(R) \|u\|_\infty^k \|v\|_\infty^\ell \|\psi\|_\infty^m.
$$
Then, the proof goes exactly along the  lines of  the proof of Proposition~\ref{prop:boundary-L1}.
\end{proof}

\section{H\"older continuous subsolution theorems}
\label{sec:holder-ss}

In this section we prove Theorem~\ref{thm:holder-ss}. 
We first show that the global capacity is equivalent to the Bedford-Taylor capacity defined via a finite covering.
We fix  a finite covering of $\ov M$ - $\{B_i(s)\}_{i\in I} \cup \{U_j(s)\}_{j\in J}$, where $B_i(s) = B_i(x_i,s)$ and $U_j(s) = U_j(y_j, s)$ are coordinate balls and coordinate half-balls centered at $x_i$ and $y_j$ respectively, and  of radius $0<s<1$. We choose $s>0$ so small that $B_i(2s)$ and $U_j(2s)$ are still contained in holomorphic charts. 
For any Borel set $E \subset M$ we can define another capacity
\[
	cap'(E) = \sum_{i\in I} cap(E\cap B_i(s), B_i(2s)) + \sum_{j\in J} cap(E\cap U_j(s), U_j(2s)), 
\]
where $cap(\bullet,\bullet)$ on the right hand side is just the Bedford-Taylor capacity.

\begin{prop}  \label{prop:two-capacities}
Two capacities $cap_\om$ and $cap'$ are equivalent. More precisely, there exists a uniform constant $A_0>0$ such that for any Borel set $E\subset M$,
\[
	\frac{1}{A_0} cap'(E) \leq cap_\om (E) \leq A_0 cap'(E).
\]
\end{prop}

\begin{proof} The proof is an adaptation from \cite{Ko03}. We first prove for a uniform $C>0$,
\[ cap_\om(E) \leq C cap'(E).\]   
Let $i\in I \cup J$, and let $U(s)$ be either $B_i(s)$ or $U_i(s)$. Assume that  $\om \leq dd^c g$ for a strictly plurisubharmonic function $g \leq 0$ on a neighborhood of $\ov U(2s)$. Consider $v \in PSH(M, \om)$ with $-1\leq v \leq 0$. 
$$\begin{aligned}
	\int_{E \cap U(s)} (\om + dd^cv)^n  
&	\leq \int_{E \cap U(s)} (dd^c (g+ v))^n \\
&	\leq (\| g \|_\infty+1)^n cap(E \cap U(s), U(2s)).
\end{aligned}$$
Since $|I \cup J|$ is finite, the first inequality follows from  the sub-additivity of $cap_\om(\cdot)$. 
The inverse inequality will follow if one can show that for a fixed $i \in I \cup J$ and $U(s)$ as above there is a uniform $C>0$ such that 
\[
	cap(E \cap U(s), U(2s)) \leq C cap_\om(E).
\]
In fact,
let $v \in PSH(U(2s))$ and $-1 \leq v \leq 0$ in $U(2s)$. We wish to find  $\tilde v$  such that $\tilde v = a v -a$ on $U(s)$, $-1\leq \tilde v \leq 0$ and $ \tilde v \in PSH(M,\om)$, where $0<a<1/2$ is a uniform constant depending only on $M,\om$.

First we take a smooth function $\eta$ such that $\eta =0$ on $\ov M\setminus U(2s)$ and $\eta <0$ in $U(2s)$. Consider the function $\epsilon \eta$ for $\epsilon>0$ small. 
Since $\om$ is a Hermitian metric on $\ov M$, we can choose $\epsilon>0$ depending on $\eta$ and $\om$ such that $\epsilon \eta \in PSH(M, \om)$.  Choose $0< a <1/2$ so that $\epsilon\eta \leq -3 a$ on $U(s)$. Writing $\eta$ for $\epsilon \eta$, we conclude that there exists a smooth $\om$-psh function $\eta = 0$ on $M\setminus U(2s)$ and $\eta \leq -3a$ on $\bar U(s)$ for $0< a <1/2$.

The function $\tilde v$ is defined as follows:
\[
\tilde v :=
\begin{cases}
	\max\{a v - a, \eta\} \quad &\mbox{on} \quad U(2s), \\
	0 \quad &\mbox{on}\quad M\setminus U(2s).
\end{cases}
\]
As $\limsup_{z\to \zeta}v(z) \leq -a <  \eta(\zeta) =0$ for $\zeta \in \d U(2s) \cap M$, $\tilde v \in PSH(M, \om)$.  It is easy to see that $\tilde v$ satisfies all requirements. Thus,
\begin{align*}
	\int_E (\om + dd^c \tilde v)^n 
&\geq 	\int_{E\cap U(s)} (\om + a dd^c v)^n  \\
&\geq 	a^n \int_{E \cap U(s)} (dd^cv)^n.
\end{align*}
By taking supremum over $v$ it implies that $cap(E\cap U(s), U(2s)) \leq a^{-n} cap_\om(E)$.
\end{proof}

Using the equivalence of capacities above and local volume-capacity inequalities on boundary and interior charts (Theorem~\ref{thm:moderate-boundary} and Remark~\ref{rmk:moderate-interior}) we derive the global measure-capacity estimate on  $M$.
For $\veps>0$ small let us denote
\[\label{eq:M-epsilon}
	M_\veps = \{z \in M: {\rm dist}_\om(z, \d M) > \veps\},
\]
where ${\rm dist}_\om(\bullet, \d M)$ is the distance function on $M$ with respect to the Riemannian metric induced by $\om$.

\begin{lem}\label{lem:moderate} Let $\ul u \in PSH(M,\om) \cap C^{0,\al}(\ov M)$ for some $0<\al \leq 1$. Let $\mu$ be a positive Borel measure on $M$. Suppose $\mu \leq (\om+dd^c\ul u)^n$ in $M$. Then   there exist uniform constants $C, \al_0>0$ (also independent of $\veps$) such that for every compact set $K \subset M_\veps$, 
$$\mu (K) \leq \frac{C}{\veps^{n+1}} \exp\left( - \frac{\al_0 \veps}{[cap_\om(K)]^\frac{1}{n}} \right).$$ 
In particular, for any $\tau>0$,
$$
	\mu(K)  \leq 	 \frac{C_\tau}{\veps^{n+1}(\al_0\veps)^{n(1+\tau)}}  [cap_\om(K)]^{1+\tau},
$$
where $C_\tau$ depends additionally on $\tau.$
\end{lem}

\begin{proof}  Cover $\ov M$ by  finitely many coordinate balls $B_i(R/2)$ and coordinate half-balls $U_j(R/2)$ with $R>0$ (fixed) so that $B_i(2R)$ and $U_j(2R)$ are still contained in holomorphic charts of $\ov M$. Let $K \subset M_\veps$ be a compact set. 

Consider its part $K_i = K \cap B_i(R/2)$ which is contained in an interior chart $B_i(2R)$. On this chart we can choose a strictly plurisubharmonic function $g$ such that $dd^c g \geq \om$.  Set $\psi = g + \ul u$. Then, $\mu \leq (dd^c\psi)^n$ on $B_i(2R)$. Subtracting a constant we may assume that $\psi \leq 0$ on $B_i(2R)$. Applying Remark~\ref{rmk:moderate-interior} and  Proposition~\ref{prop:two-capacities} we get 
$$\begin{aligned}
	\mu(K_i) 
&\leq C \exp\left(\frac{-\tau_0}{[cap(K_i,B_{R})]^\frac{1}{n}} \right) \\
&\leq C  \exp\left(\frac{-\tau_0}{[A_0cap_\om(K_i)]^\frac{1}{n}} \right) \\
&\leq C  \exp\left(\frac{-\al_0}{[cap_\om(K)]^\frac{1}{n}} \right), 
\end{aligned}$$
where $\al_0 =\tau_0/A_0^\frac{1}{n}$ is a uniform constant and we used $cap_\om(K_i) \leq cap_\om(K)$ in the last inequality. 

Next consider  $K_j = K \cap U_j(R/2)$ which is  contained in a boundary chart $\Om = U_j(2R)$. Similarly as above $\mu \leq (dd^c\psi)^n$ on $\Om$ for a negative H\"older continuous $\psi $ on $\ov \Om$ which is plurisubharmonic in $\Om.$ Now Theorem~\ref{thm:moderate-boundary} and Proposition~\ref{prop:two-capacities} give
$$\begin{aligned}
	\mu(K_j) &\leq \frac{C}{\veps^{n+1}} \exp\left(\frac{-\tau_0 \veps}{[cap(K_j,\Om_R)]^\frac{1}{n}} \right)\\ &\leq \frac{C}{\veps^{n+1}} \exp\left(\frac{-\tau_0 \veps}{[cap'(K)]^\frac{1}{n}} \right) \\
&\leq \frac{C}{\veps^{n+1}} \exp\left(\frac{-\tau_0 \veps}{A_0^\frac{1}{n}[cap_\om(K)]^\frac{1}{n}} \right). \\
\end{aligned}$$
Since $|I \cup J|$ is finite, we conclude 
$$\begin{aligned}
	\mu(K) &\leq \sum_{i\in I} \mu(K_i) + \sum_{j\in J} \mu(K_j) \\&\leq  C _1 \exp\left(\frac{-\al_0}{[cap_\om(K)]^\frac{1}{n}} \right) +  \frac{C_2}{\veps^{n+1}} \exp\left(\frac{-\al_0 \veps}{[cap_\om(K)]^\frac{1}{n}} \right) \\ &\leq   \frac{C}{\veps^{n+1}}  \exp\left(\frac{-\al_0 \veps}{[cap_\om(K)]^\frac{1}{n}} \right).
\end{aligned}$$ 
This is the desired estimate.
\end{proof}

We now fix the notation  to finish  the proof of Theorem~\ref{thm:holder-ss}.
Let $\mu$ be a positive Borel measure on $M$. Suppose that there exists $\ul u \in PSH(M, \om) \cap C^{0,\al}(\ov M)$ with $0<\al \leq 1$,  a H\"older continuous subsolution for $\mu$ on $M$, satisfying  $$\ul u_{|_{\d M}} = \vphi \in C^{0,\al}(\d M).$$
By Theorem~\ref{thm:bounded-ss} and  Corollary~\ref{cor:continuity-dns} there exists a solution $u \in PSH(M, \om) \cap C^0(\ov M)$ solving
\[\label{eq:MA} (\om + dd^c u)^n = \mu, \quad \lim_{z \to q} u(z) = \vphi(q) \quad\text{for every } q \in \d M.
\]

\begin{prop}\label{prop:stability-ext}  Let $u \in PSH(M, \om) \cap L^\infty(\ov M)$ be the solution to \eqref{eq:MA}.  Let $v \in PSH(M,\om)\cap L^\infty(\ov M)$ be such that $v = u$ on $M \setminus M_\veps$. Then there is $0< \al_2 \leq 1$ such that
$$
	\sup_M (v -u) \leq \frac{C}{\veps^{3n+1}} \left( \int_M \max\{v-u,0\} d\mu \right)^{\al_2},
$$
where $C = C (M,\om, \ul u, \|u\|_\infty, \|v\|_\infty)$ and $\al_2 = \al_2(n, \al)>0$ are uniform constants.
\end{prop}

\begin{proof} Subtracting a  constant and then dividing both sides of the inequality by $(1+\|u\|_\infty + \|v\|_\infty)$ we may assume that $$-1 \leq u, v \leq 0 \quad\text{on } \ov M.$$ 
Let us assume also that $	-s_0 =  \sup_{M} (v -u) >0$, otherwise the statement trivially  follows. 

We will make use of Theorem~\ref{thm:kappa}. There for a given  number  $0<\te <1$ we defined 
$$
\te_0:=	\frac{1}{3}\min\left\{
	\te^n, 
	\frac{\te^3}{16 \rB}, 
	4 (1-\te) \te^n, 
	4 (1-\te)\frac{\te^3}{16 \rB}, |s_0| \right\} ,
$$
 $m(\te) = \inf_{M} [ u - (1-\te) v]$ and $U(\te, t ) := \{ u < (1- \te) v+ m(\te) + t \}$ for $0< t < \te_0$. Then, for $0<\te < |s_0|/3$ and $0< t< \te_0$ we have
$$s_0 + \te + 2t \leq 0.$$
Since $-1 \leq u, v \leq 0$, it is clear that
$$
	s_0 - \te \leq m(\te) \leq s_0.
$$
It follows that
$$
	U(\te, 2t) \subset \{u< v + s_0 + \te + 2t\} \subset M_\veps,
$$
where  in the last inclusion the assumption $v = u$ on $M\setminus M_\veps$ is used.

Now the proof  is identical to the one in \cite[Proposition~2.4]{KN4} except that we need to replace \cite[Lemma~2.6]{KN4} there by
Lemma~\ref{lem:moderate} at
the cost of extra factor $\veps^{-n(2+\tau)-1}$ in the uniform constants.  If we  fix $\tau =1$ in that lemma then 
$$
	\mu(K)  \leq 	 \frac{C_\al [cap(K,\Om)]^{2}}{(\al_0\veps)^{3 n+1}}.  
$$
Therefore as in  \cite[Proposition~2.4]{KN4}  we can now take $$\al_2 = \frac{1}{1+ (n+2) (n+1)}.  \qedhere$$
\end{proof}

The next step is to estimate $L^1(d\mu)$-norm in terms of $L^1(dV_{2n})$-norm. Again this estimate is obtained for functions that are equal outside $M_\veps$.

\begin{lem}\label{lem:holder-L1} Let $u$ be the solution to the equation \eqref{eq:MA}. Let $v \in PSH(M,\om)\cap L^\infty(\ov M)$ be such that $v = u$ on $M \setminus M_\veps$. Then there is $0< \al_3 \leq 1$ such that
$$
\int_M |u-v| d\mu \leq \frac{C}{\veps^n}\left( \int_M |u-v| dV_{2n}\right)^{\al_3},$$
where $C = C (M,\om, \ul u, \|u\|_\infty, \|v\|_\infty)$ and $\al_3 = \al_3(n, \al)>0$ are uniform constants.
\end{lem}

\begin{proof} As usual we cover $\ov M$ by a finite number of coordinate balls and  half-balls of radius $R/2$ so that the ones with radius $2R$ are still contained in holomorphic charts. 
On a local coordinate chart $V$ consider   a strictly plurisubharmonic function $g \leq 0$  such that
$$	dd^c g \geq \om.
$$
Then we write $u' = u+g$, $v' = v+g$ and $\psi = \ul u +g$. They are plurisubharmonic functions on $V$. Moreover, 
$$
	\mu  \leq (\om+dd^c \ul u)^n \leq (dd^c \psi)^n \quad\text{on } V,
$$
where $\psi \in C^{0,\al}(\ov V)$ is a H\"older continuous plurisubharmonic function on $V$.

On an interior chart $B_{2R}$ of $\ov M$ by Lemma~\ref{lem:interior-L1} we have
$$
	\int_{B_{R/2}}|u-v| d\mu \leq \int_{B_{R/2}} |u'-v'| (dd^c \psi)^n \leq C\left( \int_{B_R}|u'-v'| dV_{2n} \right)^{\al'}.
$$
This is bounded by 
$
	C \left( \int_M |u-v| dV_{2n} \right)^{\al'}.
$

We now consider the case of a boundary chart $\Om$. 
Let $\rho$ be is the defining function of $\d M$ on $\Om$ as in Section~\ref{sec:boundary-chart}. Then, there exists a uniform constant $c_0>0$ such that
$$
	|\rho(z)| \geq c_0 {\rm dist}_\om(z, \d M) \quad \text{for all } z\in \Om
$$
(shrinking $\Om$ if necessary so that ${\rm dist}_\om$ is a smooth function in $\Om$).
Hence $u=v$ on $\{|\rho(z)| \leq c_0\veps\}$ by the assumption. Without loss of generality we may assume $c_0 =1$. The conclusion follows from  Proposition~\ref{prop:boundary-L1}. 

Since the covering is finite, the proof of the lemma follows.
\end{proof}

We now proceed to find   the H\"older exponent of  the solution $u$ of the equation \eqref{eq:MA} over $\ov M$. By subtracting a uniform constant we may assume that 
$$	u \leq 0 \quad\text{on }\ov M.
$$

Let $\de >0$ be small. For $z \in M_\de$ we define
\[\label{eq:u-delta-max}
	\wh u_\de (z) = \sup \{u(x): x\in M \quad \text{and} \quad{\rm dist}_\om(x, z) \leq \de\}.
\]
where ${\rm dist}_\om$  is the Riemannian distance induced by the metric $\om$. 
We wish to show that there exist constants $c_0>0, \de_0>0$ and an exponent $0<\tau \leq 1$  such that for every $0<\de \leq \de_0$,
\[\label{eq:holder-norm}
	\sup_{\ov M_\de} (\wh u_\de - u) = \sup_{M_\de} (\wh u_\de - u) \leq c_0 \de^{\tau}.
\]

To do it let us fix a small constant $\de_0>0$ such that for every $0< \de < \de_0$,
$
	\sup_{M_\de} (\ul u_\de - \ul u) \leq C \de^{\al}
$
 (for which we use H\"older continuity of the subsolution). Now we consider two parameters $\de, \veps$ such that 
\[ \label{eq:parameters} 0< \de \leq \veps < \de_0.
\]

By a classical argument we obtain the following estimate on the $(\veps,\de)$-collars near the boundary $\d M$. 

\begin{lem} \label{lem:boundary-value} Consider $0< \de \leq \veps < \de_0$. There is $0< \tau_1 \leq 1$ such that for $z \in \ov{M_\de}\setminus M_\veps$, 
$$
	\wh u_\de(z) \leq u(z) + c_1 \veps^{\tau_1},
$$
where $\tau_1$ depends only on $M, \om$ and  $\al $ the H\"older exponent of $\ul u$.
Moreover, for $z \in M$ and $q \in \d M$ with ${\rm dist}_\om(z, q) \leq \de$ we have
$$|u(z) - u(q)| \leq c_1 \de^{\tau_1}.$$
\end{lem}

\begin{proof}Let $h_1 \in C^0(\ov{M}, \bR)$ be the unique solution to the linear PDE:
\[\label{eq:omega-laplace}\begin{aligned}
	(\om + dd^c h_1) \wed \om^{n-1} =0, \\
	h_1 = \vphi \quad\text{on } \d M.
\end{aligned}
\]
Note that $(\om +dd^c u)\wed \om^{n-1} \geq 0$. The maximum principle for the Laplace operator with respect to $\om$  gives $u \leq h_1.$ It is classical fact from \cite[Theorem~6]{JY93} (see also \cite[Theorem~5.3]{SY}) that $h_1$ is also H\"older continuous on $\ov M$:   
$$
	h_1 \in C^{0, \tau_1}(\ov M),
$$
with $0<\tau_1 \leq \al$ (decreasing $\tau_1$ if necessary).
Fix $z \in \ov{M_\de}\setminus M_\veps$. Since $u$ is continuous, there exists a point $z_\de \in M$ with ${\rm dist}_\om(z_\de, z) \leq \de$ such that $u(z_\de) = \wh u_\de(z)$. Let $q \in \d M$ be the point that is closest to $z$, i.e., ${\rm dist}_\om(z, q) \leq \veps$. By the fact that $\ul u \leq u \leq h_1$ on $M$ we have $$\wh u_\de(z) - u(z) = u(z_\de) - u(z) \leq h_1(z_\de) - \ul u(z).$$  Since $h_1(q) = \vphi(q) = \ul u(q)$, it follows that
$$\begin{aligned}
	h_1(z_\de) - \ul u(z) 
&= (h_1(z_\de) - h_1(q)) + (\ul u(q) - \ul u(z))  \\
&\leq (h_1(z_\de) - h_1(z)) + (h_1(z) - h_1(q)) + (\ul u(q) - \ul u(z))\\
&\leq C (\de^{\tau_1} + \veps^{\tau_1}),
\end{aligned}$$
where $C$ is maximum of  the H\"older norms of $h_1$ and $\ul u$ on $\ov M$. Because $\de \leq \veps$,
the proof of the first inequality is completed.
The remaining inequality  follows from a similar argument.
\end{proof}

\begin{remark} The constants $c_1 , \tau _1$ in the last lemma are independent of parameters $\de, \veps$. This applies also to all uniform constants appearing in the  estimates that follow.
\end{remark}

\begin{cor}\label{cor:holder-sufficient}  The H\"older continuity of $u$ over $\ov M$ will follow once we prove  \eqref{eq:holder-norm}.
\end{cor}

\begin{proof}
We need to justify that for every $x, y \in \ov M$,
$$
	|u(x) - u(y)| \leq C {\rm dist}_\om(x,y)^{\tau},
$$
where $\tau = \tau _1$ in Lemma~\ref{lem:boundary-value}. We may assume that $0< {\rm dist}_\om(x,y) = \de \leq \de_0$,  and ${\rm dist}_\om(x,\d M) = \de_x \leq \de_y = {\rm dist}_\om(y,\d M)$. 

{\bf Case 1:} $\de_y \leq \de$. Then there exist $q_x, q_y \in \d M$ such that ${\rm dist}_\om(x,q_x) \leq \de$ and ${\rm dist}_\om(y,q_y) \leq \de$.  Since  ${\rm dist}_\om(q_x, q_y) \leq 3 \de$, by the H\"older continuity of $\ul u$ we have $|u(q_x) - u(q_y)| = |\ul u(q_x)-\ul u(q_y)| \leq C \de^{\tau}$. It follows from Lemma~\ref{lem:boundary-value} that
$$\begin{aligned}
	|u(x) - u(y)| &\leq |u(x) - u(q_x)| + |u(y) - u(q_y)| + |u(q_x)-u(q_y)| \\&\leq C{\rm dist}_\om(x,y)^{\tau}. 
\end{aligned}
$$

{\bf Case 2:} $\de_x \geq \de$.  Then, $u(y) \leq \wh u_\de(x)$ and $u(x) \leq\wh u_\de(y)$. Therefore, $u(x) - u(y) \leq \wh u_\de (y) - u(y) \leq C \de^\tau$ by \eqref{eq:holder-norm}. Similarly, $u(y) - u(x) \leq C \de^\tau$.

{\bf Case 3:} $\de_x < \de < \de_y$. Without loss of generality we may assume that $\de$ is a regular value of ${\rm dist}_\om(\bullet, \d M)$; that means: there exists a point $p \in \d M_\de$, which is the intersection of the shortest path joining $x, y$ and $\d M_\de$, such that $$ \max\{{\rm dist}_\om(x,p), {\rm dist}_\om(y,p)\} \leq \de.$$ Hence,
$$
	|u(x) - u(y)| \leq |u(x) - u(p)| + |u(y) -u(p)| \leq C{\rm dist}_\om(x,y)^{\tau},
$$
where we used Case 1 and Case 2  for the  last inequality.
\end{proof}

Now we use the global regularization of a quasi-plurisubharmonic function due to Demailly \cite{De94} which provides a lower bound on the complex Hessian.

Let $u$ be the continuous solution to the equation~\eqref{eq:MA}. Consider 
\[\label{eq:phie}
\rho_\delta u(z)=\frac{1}{\delta ^{2n}}\int_{\zeta\in T_{z}M}
u({\exp} h_z(\zeta))\chi\Big(\frac{|\zeta|^2_{\omega }}{\delta ^2}\Big)\,dV_{\omega}(\zeta), \quad z\in M_\de, \ \delta>0;
\]
where $\zeta \to {\exp}h_z(\zeta)$ is the (formal) holomorphic part of the Taylor expansion of the exponential map of  the Chern connection on the tangent bundle of $M$ associated to $\omega $, and the mollifier 
 $\chi: \mathbb R_{+}\rightarrow\mathbb R_{+}$  is given by \eqref{eq:molifier} above. 

By \cite[Remark~4.6]{De94} we have for $0< t < \de_0$ small
\[\label{eq:reg-loc} 
\begin{aligned}
	\rho_t u (z) 
&= 	\int_{|x|<1} u(z + t x) \chi (|x|^2) dV_{2n}(x) + O(t^2) \\
&=	u * \chi_t(z) + O(t^2)
\end{aligned}\]
in the normal coordinate system centered at $z$ (see \cite[Proposition~2.9]{De94}). 
Since the metric $\om$ is smooth on $\ov M$, the distance function satisfies
$${\rm dist}_\om(z, x) = |z-x| + O(|z-x|^2)$$  as $x\to z.$
 It follows that for $0< \de \leq \de_0 $ ($\de_0$ small enough)
\[\label{eq:regularization-supremum}
	\rho_\de u(z) \leq \sup_{|z-x| \leq \de} u(x) + C \de^2 \leq  \wh u_{2\de} (z) + c_2 \de \quad \text{for every } z \in \ov{M_\de}.
\]

We observe that to prove the H\"older continuity it is enough to work with this geodesic convolution.
\begin{lem}\label{lem:equiv-holder} Suppose that there exist constants $ 0< \tau \leq 1$ and $C>0$ such that for $0< \de \leq \de_0$,
\[
	\rho_{\de} u(z) - u(z) \leq C \de^{\tau} \quad\text{for every } z\in M_\de.
\]
Then  $u$  is H\"older continuous on $\ov M$.
\end{lem}

\begin{proof} By Corollary~\ref{cor:holder-sufficient} it is enough to verify the inequality \eqref{eq:holder-norm}, i.e, $$\sup_{M_\de} (\wh u_\de -u) \leq C' \de^\tau.$$ First, it follows from  \eqref{eq:reg-loc} that in normal coordinates containing ball of radius $\de$ we have $u *\chi_\de - u \leq C \de^\tau$. Therefore, \eqref{eq:cov-reg2} implies  $\cw u_{\de/2}  -u \leq C \de^\tau.$ Now, the proof of the required inequality
follows the  lines of   argument given in \cite[Lemma~4.2]{GKZ08} (see also \cite[Theorem~3.2]{Z21}).
\end{proof}

Now let us state the important estimate for the complex Hessian of $\rho_\de u$.
The proof of the following variation of \cite[Proposition~3.8]{De94} and \cite[Lemma 1.12]{BD12} was given in \cite[Lemma~4.1]{KN2}.

\begin{lem}\label{lem:kis}
 Let $0<\de < \de_0$ and $\rho_tu$ be as in \eqref{eq:phie}. 
 Define the Kiselman-Legendre transform with level $b>0$ by
 \begin{equation}\label{kisleg}
 u_{\delta , b } (z)= \inf _{ t\in [0,\delta ]}\left(\rho_{t }u (z)+ Kt^2  + Kt -b \log\frac{t}{\delta } \right),
 \end{equation}
Then for  some positive constant $K$  depending on the curvature, the function $\rho_{t } u+Kt^2$ is increasing in $t$ and
 the following estimate holds:
\begin{equation}\label{hessest}
\omega+dd^c  u_{\delta,b }\geq -(A b+2K\delta)\,\omega \quad \text{on } \ov{M_\de}.
\end{equation}
where $A$ is a lower bound of the negative part of the Chern curvature of $\omega$.
\end{lem}

Thanks to this lemma we  construct an $\om$-psh function $U_\de$ which approximates the solution $u$. 

\begin{lem}\label{lem:U-delta} Let $0< \tau \leq 1$ and $b = (\de^\tau - 2K\de)/A = O(\de^\tau)$, where $K, A, \de$ are parameters in Lemma~\ref{lem:kis}. Define $U_\de :=  (1-\de^\tau) u_{\de,b}.$ Then, $U_\de \in PSH(M_\de, \om)$ satisfies
$$
	U_\de  \leq u + c_3 \de^\tau + c_1 \veps^{\tau_1} \quad\text{on }\d M_\veps.
$$
Moreover, 
$$
	U_\de \leq \rho_\de u + c_3 \de^\tau \quad\text{on } \ov{M_\de},
$$
where $c_1,c_3$ are uniform constants independent of $\veps$ and $\de$.
\end{lem}

\begin{proof}
By Lemma~\ref{lem:kis} we have $\om + dd^c U_\de \geq \de^{2\tau} \om.$ 
The monotonicity of $\rho_t u + Kt^2$ implies 
\[\label{eq:kis-sup}	u  \leq u_{\de, b} \leq \rho_\de u + 2 K\de \quad\text{on } \ov{M_\de}.
\]
On the boundary $\d M_\veps$, using \eqref{eq:regularization-supremum} and Lemma~\ref{lem:boundary-value} (with $\de' = 2\de < \veps$), we have 
\[\label{eq:kis-sup1}
	\rho_\de u 
\leq \wh u_{2\de}  + c_2 \de  
\leq u + c_1\veps^{\tau_1} + c_2 \de.
\]
Since $u_{\de,b} \leq 0$ is uniformly bounded, it is easy to see that 
\[\label{eq:kis-sup1a}(1-\de^\tau) u_{\de,b} \leq u_{\de, b} + C \de^\tau.\] By \eqref{eq:kis-sup}, \eqref{eq:kis-sup1} and \eqref{eq:kis-sup1a}  we have on $\d M_\veps,$
\[\begin{aligned}
	U_\de &\leq \rho_\de u + 2K\de + C  \de^\tau \\
	&\leq u + c_1 \veps^{\tau_1} + (2K+c_2) \de + C \de^{\tau}. 
\end{aligned}\]
The proof of on the boundary part is completed by taking  $c_3 = 2K + c_2 + C$.
Furthermore, combining \eqref{eq:kis-sup} and \eqref{eq:kis-sup1a} we get  
$U_\de \leq \rho_\de u + c_3 \de^\tau$ on  $\ov{M_\de}.$
\end{proof}

By this lemma and the stability estimates we get the following bound.

\begin{prop} \label{prop:delta} Let $0<\de \leq \veps \leq \de_0$. Let $U_\de$ be the function defined in Lemma~\ref{lem:U-delta}. Then, 
$$	\sup_{M_\de} (U_\de - u) \leq \frac{C \de^{\al_2\al_3}}{\veps^{4n+1}} +  c_1\veps^{\tau_1},
$$
where the exponents $\al_2, \al_3>0$ and $\tau_1$ are from  Proposition~\ref{prop:stability-ext}, Lemma~\ref{lem:holder-L1}, and Lemma~\ref{lem:boundary-value}, respectively.
\end{prop}

\begin{proof} Using Lemma~\ref{lem:U-delta} we can  produce:
\[ \label{eq:U-hat} \wt U_\de = \begin{cases}
\max\{U_\de - c_1\veps^{\tau_1} - c_3\de^\tau, u\}   &\quad\text{on } M_\veps, \\
u &\quad\text{on } M\setminus M_\veps ,
\end{cases}
\]
which  is an $\om$-psh function on $M$.
 Next we  apply Proposition~\ref{prop:stability-ext} to get that for some $0<\al_2 \leq 1$,
\[\label{eq:delta1}
	\sup_M(\wt U_\de - u) \leq \frac{C}{\veps^{3n+1}} \left\|(\wt U_\de - u)_+\right\|_{L^1(d\mu)}^{\al_2}.
\]
Note that $0< \de\leq \veps$, so $M_\veps \subset M_\de$. Then,
$$\begin{aligned}
	(\wt U_\de - u)_+ &= \max\{ U_\de - c_1\veps^{\tau_1} - c_3\de^\tau-u,0\} \cdot {\bf 1}_{M_\veps} \\&\leq \max\{\rho_\de u -u, 0\} \cdot  {\bf 1}_{M_\veps},
\end{aligned}$$
where we used $U_\de - c_3\de^\tau - u \leq \rho_\de u - u$  on $\ov M_\de$ (Lemma~\ref{lem:U-delta}) for the second inequality.
Hence,
\[\label{eq:delta1b}\notag
(\wt U_\de-u)_+ = {\bf 1}_{M_\de}  \cdot (\wt U_\de-u)_+  \leq   {\bf 1}_{M_\de} \cdot  (\rho_\de - u)_+.\]
This combined with Lemma~\ref{lem:holder-L1} gives for some $0<\al_3 \leq 1$,
\[\label{eq:delta2}\begin{aligned}
	 \| (\wt U_\de -u)_+\|_{L^1(d\mu)} 
&\leq		\frac{C}{\veps^n} \| (\wt U_\de -u)_+\|_{L^1(dV_{2n})}^{\al_3} \\
&\leq 	\frac{C} {\veps^n} \| {\bf 1}_{M_\de} \cdot( \rho_\de -u)_+\|_{L^1(dV_{2n})}^{\al_3}.
\end{aligned}
\] 

The next step is to show that $L^1$-norm on the right hand side has a desired bound. 
Covering $M_{2\de}$ be finitely many normal charts (contained on coordinates balls or coordinate half-balls) and invoking the inequalities \eqref{eq:reg-loc}, \eqref{lem:cov-reg-compare}   and  Lemma~\ref{lem:boundary-L1} one obtains  that for $0<\de < \de_0$,
$$
	\int_{M_{2\de}} (\rho_\de u - u)_+ dV_{2n} \leq C \de^2 +  C\de.
$$
This implies
\[\label{eq:delta3}
	\int_{M_\de} (\rho_\de -u)_+ dV_{2n} \leq C \int_{M_{\de}\setminus M_{2\de}} dV_{2n} + \int_{M_{2\de}} (\rho_\de - u)_+  dV_{2n} \leq C\de,
\]
where we used the compactness of $\ov M$ to get $V_{2n} (M_{\de} \setminus M_{2\de}) \leq C\de$.

We conclude from \eqref{eq:delta1}, \eqref{eq:delta2}, \eqref{eq:delta3} and $0<\al_2,\al_3\leq 1$ that
\[\label{eq:delta4}
	\sup_M  (\wt U_\de - u)   \leq \frac{C \de^{\al_2\al_3}}{\veps^{4n+1}} .
\]
Notice that 
$$\begin{aligned}
\sup_{M_\de} (U_\de - u) 
&\leq \max\left\{ \sup_{M_\veps} (U_\de - u), \sup_{\ov M_\de\setminus M_\veps} (U_\de - u) \right \}\\
&\leq  \max\left\{ \sup_{M} (\wt U_\de - u), \sup_{\ov M_\de\setminus M_\veps} (U_\de - u) \right\}.
\end{aligned}$$
Combining this with \eqref{eq:delta4} and Lemma~\ref{lem:boundary-value} we get
$$\begin{aligned}
	\sup_{M_\de}(U_\de - u) 
&\leq 	\sup_{M} (\wt U_\de - u) +  c_1\veps^{\tau_1} \\ 
&\leq 	\frac{C \de^{\al_2\al_3}}{\veps^{4n+1}} + c_1\veps^{\tau_1}.
\end{aligned}$$
This is the desired inequality.
\end{proof}

We are ready to verify the hypothesis of Lemma~\ref{lem:equiv-holder}.
\begin{proof}[End of the proof of the H\"older continuity of $u$]
Let us choose 
$\al_4 = \al_2\al_3 / 2(4n+1)>0$  and $\veps = \de^{\al_4}$, then Proposition~\ref{prop:delta} implies 
$$
\sup_{M_\de} (U_\de -u) \leq C \de^{\al_4} + c_1 \de^{\tau_1\al_4}
\quad\text{on } \ov M_\de.
$$
Therefore it follows from $u_{\de, b} \leq 0$ and $0< \tau_1, \al_4 \leq 1$ that
\[\label{eq:end1}
	u_{\de,b} - u \leq   U_\de - u \leq C \de^{\tau_1\al_4} \quad\text{on } \ov M_\de.
\]

Let us fix a point $z\in \ov M_\de$. The minimum in the definition of $u_{\de, b}$ is realized at $t_0 = t_0(z)$. So, at this point we have
$$
	\rho_{t_0}u + Kt_0^2 + Kt_0 - b \log(t_0/\de) - u \leq C \de^{\tau_1\al_4}.
$$
Since $\rho_{t}u+ Kt^2 + Kt - u \geq 0$, we have 
$b \log(t_0/\de) \ge - C \de^{\tau_1\al_4}. $ Hence,
$$
	t_0 \geq e^{-C \de^{\tau_1\al_4}/b}\; \de,
$$
where $$\frac{\de^{\tau_1\al_4}}{b} =  \frac{A \de^{\tau_1 \al_4}}{ \de^\tau - 2K\de}. $$
	Now we choose $\tau = \tau_1\al_4>0$, so $b\geq \de^{\al_1\al_4}/2A$, which implies
$$
	t_0 \geq c_4 \de 
$$
with $c_4=  e^{-2CA}$ as a uniform constant.

Since $t \mapsto \rho_tu + Kt^2$ is increasing in $t$, 
$$\begin{aligned}
&	\rho_{c_4\de} u(z) + K(c_4\de)^2 + K (c_4\de) - u(z) \\
&\leq 	\rho_{t_0}u(z) + Kt_0^2 + Kt_0 - b \log (t_0/\de)- u(z) \\
&= 		u_{\de, b}(z) - u (z) \\
&\leq C \de^\tau,
\end{aligned}$$
where we used the definition of $t_0$ and \eqref{eq:end1} for the third identity and the last inequality, respectively.
Hence,  $\rho_{c_4\de}u(z) - u(z) \leq C \de^{\tau } \leq C\de^\tau$. The desired estimate is obtained by rescaling $\de:=c_4\de$.
\end{proof}

Let us give the proof of uniqueness of solutions on Stein manifolds or when $\om$ is  K\"ahler.  We will need a variation of Lemma~\ref{lem:mCP} for the proof of uniqueness on Stein manifolds. In this setting there exists a strictly plurisubharmonic function on the whole manifold.  This is a straightforward generalization of \cite[Theorem~3.1]{KN1}.

\begin{lem} \label{lem:comparison-stein}Fix $\te >0$ small. Let $u, v \in PSH(M, \om)\cap L^\infty(\ov M)$ be such that $\liminf_{z\to \d M}(u - v) \geq 0$. Suppose that $-s_0 = \sup_{M} (v-u) >0$ and $\om + dd^c v \geq \te \om$ in $M$. Then for any $0< s< \te_0 = \min\{\frac{\te^n}{16\rB}, |s_0|\}$,
$$	\int_{\{u < v +s_0 +s\}} \om_v^n  \leq \left(1+ \frac{C_n \rB s}{\te^n} \right) \int_{\{u< v +s_0+s\}} \om_u^n.
$$  
\end{lem}

\begin{proof}[Proof of Corollary~\ref{cor:uniqueness}] 
We first assume that $\om$ is K\"ahler. Let $\veps>0$ and define $u_\veps = \max\{u+\veps, v\}$. Then, $u_\veps \geq u+ \veps$ on $M$ and by the assumption we have $u_\veps = u+\veps$ near $\d M$. Moreover,  since $\om_v^n \geq \om_u^n$, it follows from a well-known inequality of Demailly that
$$\begin{aligned}
	\left(\om + dd^c \max\{u+\veps, v\}\right)^n 
	\geq {\bf 1}_{\{u+\veps \geq v\}} \om_u^n + {\bf 1}_{\{u+\veps < v\}} \om_v^n 
	\geq \om_u^n.
\end{aligned}$$
Applying 
a result of B\l ocki \cite[Theorem~2.3]{Bl09}
for $u_\veps$ and $u+\veps$ we get $u_\veps= u+\veps$ on $M$ (strictly speaking he only stated for continuous functions, but the proof works for bounded functions). Thus, $u+\veps \geq v$ on $M$ . Letting $\veps$  to zero we get the proof of the corollary in the first case.

Next, suppose that $M$ is a Stein manifold and $\om$ is a Hermitian metric on $\ov M$.  Let $\rho \in C^\infty(M)$ be a strictly plurisubharmonic exhaustive function for  $M$.
As in the previous case we only need to prove $u+ \veps \geq v$ on $M$ for a fixed $\veps>0$, then let $\veps$ tend to zero. Hence we may assume that $\liminf_{z\to \d M} (u-v)(z) \geq \veps$.

Let $M'  \subset \subset M$ be such that $u \geq v+ \veps$ on $M \setminus M'$. By subtracting a uniform constant we may assume that $u, v \leq 0$ and  $- C_0 \leq \rho \leq 0$ on $\ov{M'}$ for some constant $C_0>0$. Under these assumptions the proof follows the  lines of the one in \cite[Corollary~3.4]{KN1} after replacing $\Om$ by $M'$ and  the comparison principle by Lemma~\ref{lem:comparison-stein}.
\end{proof}

\end{document}